\documentclass[12pt,twoside]{amsart}
\usepackage[utf8]{inputenc}
\usepackage[english]{babel}
\usepackage[T1]{fontenc}
\usepackage{amssymb,amsmath,amsthm,mathtools,enumitem}
\usepackage{color}
\usepackage{subfigure}
\usepackage[mathscr]{eucal}
\usepackage[version=4]{mhchem}
\baselineskip

\theoremstyle{definition}
\newtheorem{remark}{Remark}
\newtheorem{definition}{Definition}

\theoremstyle{theorem}
\newtheorem{theorem}{Theorem}
\newtheorem{corollary}{Corollary}
\newtheorem{lemma}{Lemma}

\newcommand{\C}{\mathbb{C}}
\newcommand{\N}{\mathbb{N}}
\newcommand{\R}{\mathbb{R}}

\newcommand{\geps}{\varepsilon}

\begin{document}
	\title[Enzyme catalysis with oscillatory substrate and inhibitor supplies]{Intraspecific and monotone enzyme catalysis with oscillatory substrate and inhibitor supplies}
	\author[D\'{i}az-Mar\'{i}n H. G.]{D\'{i}az-Mar\'{i}n Homero G.*}
	\email{homero.diaz@umich.mx}
	
	\author[S\'anchez-Ponce J. L.]{S\'anchez-Ponce Jos\'e L.*}
	\email{0001174k@umich.mx}
	\address{*Facultad de Ciencias F\'{i}sico-Matem\'aticas, Universidad Michoacana. Edificio Alfa, Ciudad Universitaria, C.P. 58040. Morelia, Michoac\'an, M\'exico.}

\begin{abstract}
Enzyme catalysis in reactors for industrial applications usually require an external intervention of the species involved in the chemical reactions. We analyze the most elementary open enzyme catalysis with competitive inhibition where a time-dependent inflow of substrate and inhibitor supplies is modeled by almost periodic functions. We prove global stability of an almost periodic solution for the non-autonomous dynamical system arising from the mass-law action. This predicts a well behaved situation in which the reactor oscillates with global stability. {This is a first case study in the path toward broader global stability results regarding intraspecific and monotone open reaction networks.}
\end{abstract}
	
	\subjclass[2020]{Primary: 34C12, 34C27, 92C45}
	\keywords{Mass-action law, enzyme, reaction network, almost periodic functions, {cooperative systems}, monotone systems}
	\maketitle
	
%%%%%%%%%%%%%%%%%%%%%%%%%%%%%%%%%%%%%%%%%%%%

\section{Introduction}

In reaction networks the law of mass-action leads to a polynomial vector field on a positive cone $\mathbb R^d_{> 0}$ where $d\in\mathbb N$ stands for the number of different species that intervene in the chemical reaction. Moreover, due to linear conservation laws, dynamics may be reduced to polynomial vector fields restricted to a polytope $\mathsf S^\circ\subset \mathbb R^d_{>0}$ called the {\em stoichiometric space}, contained in an affine space, $\mathsf S\subset \mathbb R^d$. For a whole account of such formalism see for instance \cite{Feinberg-2019}.

We will be interested in reaction networks that are open to the influx or efflux of certain species. According to \cite{Feinberg-2019}, this considerations are motivated by models of homogenous reactors, alluding to the transport of substrates and products in and out of the reactor. A similar point of view for open reaction networks is also adopted in the formalism introduced in \cite{Baez-2017,Baez-2018, Baez-2020} for {\em open reaction networks}. Concentration inflows and outflows manifest themselves in the corresponding dynamical system as forcing terms. The examples studied in \cite{Baez-2017,Baez-2018,Baez-2020} and \cite{Feinberg-2019} incorporate constant forcing terms in the {autonomous} system.

For a more comprehensive treatment, we are interested in stability results corresponding to reaction networks with oscillatory (time-dependent) inflows and outflows, i.e. {\em non-autonomous} dynamical systems.

In the periodic case, further study of open reaction networks associated to enzyme catalysis with periodic inputs appear for instance in \cite{Boiteux, Diaz-Osuna-Lara-19, Katriel, Krupska, Stoleriu}, while in \cite{Diaz-Lopez-Osuna-2022c} we describe the almost periodic case. Other stability results in simple protein transcription, regarded as a reaction network, can be obtained for the case of periodic inputs \cite{DelVecchio}. In such work it is stressed the search of a modularity property which would allow to deduce properties of complex reaction networks by considering simple ones. These simple reactions, which may be regarded as building blocks, can be composed or concatenated by its inputs and outputs. Related to such models, in \cite{Diaz-Osuna-Villavicencio-2023} we have proved and generalized to the almost periodic case, a global stability statement properly exemplified and motivated in \cite{DelVecchio}.

As an exploratory case study for a wider project we take the typical reaction where a substrate S is catalysed by an enzyme E, obtaining a complex ES from which we extract a product P,
\[\ce{
	E + S 
		<=>[$k_1$][$k_2$]
	ES
		->[$k_3$]
	E + P
}.\]
Our goal along this work is to extend the study of global stability phenomena dealing with oscillatory inputs and / or outputs. Forcing terms may appear simultaneously without assuming any synchronicity of their frequencies. Accordingly, we consider an {\em inhibitor}, I, adding the reaction,
\begin{equation}
\ce{
	E + I 
		<=>[$k_5$][$k_4$]
	EI
}.
\end{equation}
We get an {\em open} reaction network, where we shall consider an oscillatory source introducing species S,I at a time dependent speed $F_{\rm S} (t)$, $F_{\rm I}(t)$, respectively. Thus, in agreement with \cite{Feinberg-2019} (which can be also adapted to the formalism of \cite{Baez-2017}) a third fictitious reaction is introduced as follows: 
\begin{equation}\label{eq:reaction}
\begin{aligned}
\ce{
	E + S 
		<=>[$k_1$][$k_2$]
	ES
		->[$k_3$]
	E + P
}
\\
\ce{
	E + I 
		<=>[$k_5$][$k_4$]
	EI
}
\\
\ce{
	 I 
		<=>[$\xi$_I][$F$_I$(t)$]
	$\emptyset$
		<=>[$F$_S$(t)$][$\xi$_S]
	S
}
\end{aligned}\end{equation}
where $\xi_{\rm I},\xi_{\rm S}$ are decay rates for the inhibitor and the supply substance respectively.
The corresponding numerical explorations described in \cite{Craciun-2006} for reaction \eqref{eq:reaction} exhibits a globally stable stationary state under the assumption of constant supplies $F_{\rm S},F_{\rm I}\in\mathbb R_{>0}$ in the reactor. When we consider oscillating forcing terms $F_{\rm S}(t),F_{\rm I}(t)$, in this illustration taken from \cite{Craciun-2006}, our main result claims global stability for a positive almost periodic solution.
 
More precisely, we address the dynamics under mass-law action treated by the following coupled system of ordinary differential equations,
\[
\begin{aligned}
\displaystyle
\dot{c}_{\rm S}&=
	k_2c_{\rm ES}-k_1c_{\rm E}c_{\rm S}+F_{\rm S}(t)-\xi_{\rm S}c_{\rm S},
\\
\dot{c}_{\rm I}&=
	k_4c_{\rm EI}-k_5c_{\rm E}c_{\rm I}+F_{\rm I}(t)-\xi_{\rm I}c_{\rm I},
\\
\dot{c}_{\rm E}&=
	k_2c_{\rm ES}-k_1c_{\rm E}c_{\rm S}+k_4c_{\rm EI}-k_5c_{\rm E}c_{\rm I}+k_3c_{\rm ES},
\\
\dot{c}_{\rm ES}&=
	k_1c_{\rm E}c_{\rm S}-k_2c_{\rm ES}-k_3c_{\rm ES},
\\
\dot{c}_{\rm EI}&=
	k_5c_{\rm E}c_{\rm I}-k_4c_{\rm EI},
\\
\dot{c}_{\rm P}&=
	k_3c_{\rm ES},
\end{aligned}
\]
where $c_\sigma$ stands for the concentration of the species $\sigma$. Since $c_{\rm P}$ is completely determined by the remaining differential equations, we can omit the equation corresponding to the product P. Furthermore, we reduce our problem by restricting the dynamics to the stoichiometric (affine) space
\[
	\mathsf{S}\subset\R^5,\quad
	\dim \mathsf{S}=4,
\]
obtained by the conservation law given by the enzyme global constant amount $T>0$,
\begin{equation}\label{eq:law}
	c_{\rm E}+c_{\rm EI}+c_{\rm ES}=T.
\end{equation}
Hopefully, the reduced system
\begin{equation}\label{system}
\begin{aligned}
\dot{c}_{\rm S}&=
	-k_1(T-c_{\rm ES}-c_{\rm EI})c_{\rm S}+k_2c_{\rm ES}+F_{\rm S}(t)-\xi_{\rm S}c_{\rm S},
\\
\dot{c}_{\rm I}&=
	-k_5(T-c_{\rm EI}-c_{\rm ES})c_{\rm I}+k_4c_{\rm EI}+F_{\rm I}(t)-\xi_{\rm I}c_{\rm I},
\\
\dot{c}_{\rm ES}&=
	k_1(T-c_{\rm ES}-c_{\rm EI})c_{\rm S}-(k_2+k_3)c_{\rm ES},
\\
\dot{c}_{\rm EI}&=
	k_5(T-c_{\rm ES}-c_{\rm EI})c_{\rm I}-k_4c_{\rm EI},
\end{aligned}\end{equation}
which we also denote as
\[
	\dot{c}=V(c,F(t))
\]
with parameters $k=(k_1,k_2,k_3,k_4,k_5)\in\mathbb R^5_{>0}$ and $T>0$, would have a global almost periodic attractor.

In order to write down our main result, let us set up the following notation,
\[
u_{\ast}:=\inf_{t\in \mathbb{R}}u(t) \;\; \text{and} \;\; u^{\ast}:=\sup_{t\in \mathbb{R}}u(t).
\]

\begin{theorem}\label{teo:principal}
Assume that $F_{\rm S}(t),F_{\rm I}(t)\geq 0$ are non-constant continuous strictly positive almost periodic input functions in system \eqref{system}. Then there exists a {unique} positive almost periodic solution which is a global attractor lying within the interior of the positive region of the stoichiometric space,
\[
	\mathsf{S}^\circ=\mathsf{S}\cap\mathbb{R}^4_{>0}.
\]
\end{theorem}

For the reader's convenience, we give a self contained text. Consequently, we will review some of the properties of almost periodic functions along Section \ref{sec:AP}. We will also introduce the notion of {\em intraspecific monotonic dynamical systems} that is used for the proof of Theorem~\ref{teo:principal}. Such notion was tacitly used in \cite{Korman} for the periodic non-autonomous cooperative dynamical systems in $\mathbb R^2$. We adapted such tools to prove stability results in~\cite{Diaz-Lopez-Osuna-22} and \cite{Diaz-Osuna-2023} for cooperative and competitive dynamical systems in the almost periodic case for $\mathbb R^2$. We want to stress the fact that the concept of intraspecific monotonicity is stronger than the classical monotonicity property described originally by Hirsch and Smith, see for instance \cite{Hirsch-I, Smith-Book}. Along Section \ref{sec:results} we present the proof of our main results. We finally display numerical examples in Section \ref{sec:examples} which illustrate our main statement.

{This is a first article in a series of works in preparation, such as \cite{DOS-1} coauthored by colaborators. We study as subject general {\em intraspecific and monotone open reaction networks}. We have sketched such reaction network along the introduction. Our aim consists in elucidating global stability for such reaction networks using the tools we have developed for the open enzyme catalysis studied along this exploratory work. In \cite{DOS-1} we will deal with the case of more general intraspecific and monotone open reaction networks using other kinetics instead of just the Michaelis-Menten power mass-law. Another suitable problem which we will address in a future work are the existence of functorial properties for a suitable subcategory associated to this kind of open reaction networks (regarded as open petri nets) in the categorical framework developed by Baez et al. in \cite{Baez-2017,Baez-2018,Baez-2020}.}

\section*{Aknowledgements}

{The authors thank O. Osuna for introducing us to almost periodic functions as well as to the tools developed by Korman in reference \cite{Korman} for 2-dimensional cooperative dynamical systems. We also thank the anonymous referee for carefully reading the first manuscript pointing out some mistaken calculations.}

The authors declare having no conflict of interest.

%%%%%%%%%
\section{Intraspecific monotonic dynamical systems}\label{sec:AP}

\subsection{$K$-partial ordering}

Monotone dynamical systems were introduced in \cite{Hirsch-I, Smith-Book}, we just recall some definitions.

\begin{definition}

Given $n,m\in\N$ we define the {\em orthant} $K=\mathbb R_{\geq 0}^n\times \mathbb R_{\leq 0}^m\subset \mathbb R^d$ or
\[\begin{aligned}
	K&=
	\left\{
	u\in\mathbb R^d
	\,:\, u_i\geq 0,
	u_j\leq 0
	,\, i=1,\dots n,\, j=n+1,\dots,n+m=d
	\right\}.
\end{aligned}\]
Notice that $u\in K$ if and only if
\[
	\sum_{h=1}^d
	\lambda_te^{(h)},\quad \lambda_t\geq 0,
\]
where
\[
	e^{(h)}:=(\dots,0,\epsilon_h,0\dots)\in K,
	\qquad
	\epsilon_h=\left\{
	\begin{array}{ll}
	1, &  h=1,\dots, n,
	\\
	-1, & h=n+1,\dots, n+m=d,
	\end{array}\right.
\]
is a basis of $\mathbb R^d$. Thus, the orthant $K$ is a cone which allows us to define the $K$-{\em partial ordering} in $\R^d$ as follows:
\[
	u \preceq v \Leftrightarrow v-u\in K
\]
If $	u \preceq v,$ and $u_{l_0}\neq v_{l_0} $ for some index $l_0,$ then we write down
\[
	u\prec v.
\]
If $u_l\neq  v_l,$ for every index $l=1,\dots,d,$ then we abbreviate these strict inequalities by
\[
	u\prec\prec v.
\]
\end{definition}

Take a non-autonomous system defined in an open domain $(x,y)\in U\subset \mathbb{R}^n\times\mathbb{R}^m$
\begin{equation}\label{Comp}
\begin{split}
&\dot{x}={\rm f}(t,x(t),y(t)),\\
&\dot{y}={\rm g}(t,x(t),y(t)),
\end{split}
\end{equation}
where ${\rm f, g}$ are ${\mathcal C}^1$ w.r.t. $x,$ and continuous w.r.t. $t\geq 0$. Suppose that the Jacobian $DV\left(t,{x},{y}\right)$ of system
\[
	V\left(t,{x},{y}\right)=
	\left(\begin{array}{cc}
	{\rm f}\left(t,{x},{y}\right)\\
	{\rm g}\left(t,{x},{y}\right)
	\end{array}\right)
\]
consists of the submatrices $P=\partial_x{\rm f}\left(t,{x},{y}\right)$, $Q= \partial_y{\rm f}\left(t,{x},{y}\right)$, $R=\partial_x{\rm g}\left(t,{x},{y}\right)$ and $S=\partial_y{\rm g}\left(t,{x},{y}\right)$, so that
\[
	DV(t,{x},{y})=
	\left(\begin{array}{cc}
		P& Q\\
		R& S
	\end{array}\right)
,
\]
where,
\[
\begin{aligned}
	p_{ik}&=\partial_{x_k}{\rm f}_i,
	\qquad 1\leq i\leq n,
	\quad 1\leq k\leq n,
	\\
	q_{il}&=\partial_{y_l}{\rm f}_i,
	\qquad 1\leq i\leq n,
	\quad 1\leq l\leq m,
	\\
	r_{jk}&=\partial_{x_k}{\rm g}_j, \quad
	\qquad 1\leq j\leq m,
	\quad 1\leq k\leq n,
	\\
	s_{jl}&=\partial_{y_l}{\rm g}_j,\quad
	\qquad 1\leq j\leq m,
	\quad 1\leq l\leq m,
\end{aligned}
\]
or \[
	DV(t,x,y)=
	\left(\begin{array}{ccccccc}
		{p_{11}}	& p_{12}& \cdots& p_{1n}&
		q_{11}	& \cdots	& q_{1n}
		\\
		\vdots	& \vdots	& \ddots	&  		&
		\vdots	 & \ddots	& \vdots
		\\
		p_{n1}	& p_{n2}	& \cdots	&{p_{nn}}&
		q_{n1}	& \cdots	& q_{nn}
		\\
		r_{11}	& r_{11}	& \cdots	& r_{1n}	&
		{ s_{11}}	& \cdots	& s_{1n}
		\\
		\vdots	& \vdots	& 		& 		&
		\vdots	& 		& \vdots
		\\
		\vdots	& \vdots	& 		&  		&
		\vdots	& \ddots	& \vdots
		\\
		r_{m1}	&r_{m2}& \cdots	& 		
		r_{mn}&
		s_{m1}	& \cdots	& {s_{mn}}
	\end{array}\right).
\]
\begin{definition}
A system such as \eqref{Comp} is said
to be a $K$-{mononte system} in $\mathbb{R}\times U$, if for every $t \in \mathbb{R},$ and $(x, y)\in U$, the following inequalities hold:
\begin{equation}\label{CompH}
\begin{array}{ll}
	\epsilon_{k}\cdot p_{ik}\geq 0,
	&
	\epsilon_{n+l}\cdot q_{il}\geq 0,
	\\	
	i,k=1,\dots,n;
	& l=1,\dots,m,
	\\
	i \neq k,&
	\\
	\epsilon_{k}\cdot r_{jk}\geq  0,
	&
	\epsilon_{n+l}\cdot s_{jl}\geq  0,
	\\
	k=1,\dots,n;
	&
	j,l=1,\dots,m,
	\\	
	 &
	  j\neq l.
\end{array}
\end{equation}
Furthermore, we say that \eqref{Comp} is $K$-{\em intraspecific and monotone} if
\begin{equation}\label{CompH}
\begin{array}{ll}
	\epsilon_{k}\cdot p_{ik}\geq 0,
	&
	\epsilon_{n+l}\cdot q_{il}\geq 0,
	\\	
	i,k=1,\dots,n;
	& l=1,\dots,m,
	\\
	\epsilon_{k}\cdot r_{jk}\geq  0,
	&
	\epsilon_{n+l}\cdot s_{jl}\geq  0,
	\\
	k=1,\dots,n;
	&
	j,l=1,\dots,m.
\end{array}
\end{equation}

\end{definition}

To illustrate the difference between monotone system and intraspecific and monotone take for instance, in the case of the orthant $K= \mathbb R^n_{\geq 0}\times \mathbb R^m_{\leq 0}$. For the square matrix composed of the diagonal terms of the Jacobian matrix
\begin{equation}\label{eq:Delta-n=m}
	\Delta={\rm diag}\left\{p_{11},p_{22},\dots,p_{nn},s_{11},\dots,s_{mm}\right\}
	={\rm diag}\left\{\Delta_1,\dots, \Delta_{n+m}\right\},
\end{equation}
the monotone case signed entries do not belong to the diagonal of
$
	DV(t,x,y)
$ 
for all $t\in\R$, i.e.
\[
	\left(\begin{array}{ccccccc}
		\Delta_1	& \cdots	& 	 +	&
		-	& \cdots		& -
		\\
		\vdots		& \ddots	& 	 	&
		\vdots	 & \ddots	& \vdots
		\\
		+			& \cdots	&  \Delta_n	& -
			& \cdots		& -
		\\
		-			& \cdots	&  	 -	&\Delta_{n+1}
			& \cdots		& +
		\\
		\vdots		& 		& 		&
		\vdots	& 		& \vdots
		\\
		-		&	\cdots	&  	 -	&
		+		& \cdots	& \Delta_{n+m}
	\end{array}\right),
\]
where positivity or negativity of the entries is not required to be strict.

On the other hand, intraspecific monotonicity require the {\em signed quantities in every entry of the Jacobian} as follows
\[
	\left(\begin{array}{ccccccc}
		+		& \cdots	& 	 +	&
		-	& \cdots		& -
		\\
		\vdots	& \vdots	& \ddots	& 	 	&
		\vdots	 & \ddots	& \vdots
		\\
		+			& \cdots	&  +	& -
			& \cdots		& -
		\\
		-			& \cdots	&  	 -	&+
			& \cdots		& +
		\\
		\vdots		& 		& 		&
		\vdots	& 		& \vdots
		\\
		-			& \cdots	&  	 -	&
		+		& \cdots	& +
	\end{array}\right).
\]
Meanwhile, for the reversed order defined by the inverse cone $-K= \mathbb R^n_{\leq 0}\times \mathbb R^m_{\geq 0}$, the $(-K)$-monotonicity property follows from the signed Jacobian matrix
\begin{equation}\label{DF^T-bis}
	\left(\begin{array}{ccccccc}
		\Delta_1		& \cdots	& 	 -&
		+	& \cdots		& +
		\\
		\vdots		& \ddots	& 	 	&
		\vdots	 & \ddots	& \vdots
		\\
		-				& \cdots	&  \Delta_n	& +
			& \cdots		& +
		\\
		+				& \cdots	&  	 +	&\Delta_{n+1}
			& \cdots		& -
		\\
		\vdots		& 		& 		&
		\vdots	& 		& \vdots
		\\
		+			& \cdots	&  	 +	&
		-		& \cdots	& \Delta_{n+m}
	\end{array}\right).
\end{equation}
and the intraspecific monotonicity can be verified by as the following signs for the entries of the Jacobian
\begin{equation}\label{DF^T-bis-Intra}
	\left(\begin{array}{ccccccc}
		-			& \cdots	& 	 -&
		+	& \cdots		& +
		\\
		\vdots		& \ddots	& 	 	&
		\vdots	 & \ddots	& \vdots
		\\
		-				& \cdots	&  -	& +
			& \cdots		& +
		\\
		+				& \cdots	&  	 +	&-
			& \cdots		& -
		\\
		\vdots		& 		& 		&
		\vdots	& 		& \vdots
		\\
		+			& \cdots	&  	 +	&
		-		& \cdots	& -
	\end{array}\right).
\end{equation}

\begin{definition}
We say that $(\omega(t), \zeta(t))\in\R^n\times \R^m$ is a $K$-\textit{sub-solution} of \eqref{Comp} with respect to the order $\preceq$ if 
\begin{equation}\label{Sub}
	\left({\rm f}(t, \omega(t), \zeta(t)),{\rm g}(t, \omega(t), \zeta(t))\right)
	\succeq \left(\dot{\omega}(t),\dot{\zeta}(t)\right),
	\quad
 \forall t\geq 0.
\end{equation}
Analogously, $(\Omega(t), Z(t))\in\R^n\times\R^m$ is a $K$-\textit{super-solution}  if 
\begin{equation}\label{Sup}
	\left(\dot{\Omega}(t),\dot{Z}(t)\right)
	\succeq
	\left({\rm f}(t, \Omega(t), Z(t)), {\rm g}(t, \Omega(t), Z(t))\right)
	\quad
	\forall t\geq 0.
\end{equation}
We say that a {\em pair sub-super-solution is ordered}, if  we have
\[
(\omega(t),\zeta(t))\preceq (\Omega(t),Z(t));
	\qquad
	\forall t\in\R_{\geq 0}.
\]
\end{definition}

In the case of the reversed order defined by the inverse cone $-K= \mathbb R^n_{\leq 0}\times \mathbb R^m_{\geq 0}$ we say that $(\omega(t), \zeta(t))\in\R^n\times\R^m$ is a $(-K)$-\textit{sub-solution} of \eqref{Comp} with respect to the $(-K)$-order if 
\begin{equation}\label{Sub-bis}
\begin{split}
\dot{\omega}_i&\geq {\rm f}_i(t, \omega(t), \zeta(t)),\qquad 
\dot{\zeta}_j\leq {\rm g}_j(t, \omega(t), \zeta(t)),\\
 \forall\,  i&=1,\dots,n;\quad j=1,\dots,m;\quad t \geq0.
\end{split}
\end{equation}
Analogously, $(\Omega(t), Z(t))$ is a $(-K)$-\textit{super-solution}  if 
\begin{equation}\label{Sup-bis}
\begin{split}
\dot{\Omega}_i&\leq {\rm f}_i(t, \Omega(t), Z(t)),\qquad
\dot{Z}_j\geq {\rm g}_j(t, \Omega(t), Z(t)),\\
 \forall\,  i&=1,\dots,n;\quad j=1,\dots,m;\quad t \geq 0.
\end{split}
\end{equation}

According to Kamke, condition \eqref{CompH} is an infinitesimal characterization of the basic fact that the flow of such systems preserves the partial ordering we have just defined. This property is stated as follows.

\begin{lemma}\label{lma:Comp-preserving}
Suppose that $(\omega(t),\zeta(t))\prec\prec (\Omega(t),Z(t))$ is a sub- super-solution pair with respect to the $K$-ordering of $\R^d$. Then the flow of the monotone system \eqref{Comp} preserves the order $\prec$. That is, if a solution $({\rm x}(t),{\rm y}(t))$ has initial condition such that
\begin{equation}\label{initial}
	(\omega(0),\zeta(0))\prec \prec
	({\rm x}(0),{\rm y}(0))\prec\prec
	(\Omega(0),Z(0)).
\end{equation}
Then
\begin{equation}\label{Comp-IV-1}
	(\omega(t),\zeta(t))\prec \prec
	({\rm x}(t),{\rm y}(t))\prec\prec
	(\Omega(t),Z(t))	,
	\qquad
	\forall
	t\geq 0.
\end{equation}
\end{lemma}

This claim appears for instance in \cite{Hirsch-I} where it is stated that the r\^ole of the relation $\prec\prec$ can be also performed for by $\prec,\preceq$.

\begin{proof}[Proof of Lemma \ref{lma:Comp-preserving}] {Define $w(t):={\rm x}(t)-\omega(t)$, and $z(t):={\rm y}(t)-\zeta(t)$, so that ${\rm f}(t,\omega,\zeta)\geq \dot{\omega}$ implies
\[\begin{aligned}
	\dot{w}&=
	{\rm f}(t,{\rm x},{\rm y})-\dot{\omega}
	\\
	&={\rm f}(t,\omega,\zeta)- \dot{\omega}
+
	\partial_x{\rm f}\left(t,{\rm x}^{\theta_0},{\rm y}^{\theta_0}\right)^\dagger({\rm x}-\omega) 
	+
	\partial_y{\rm f}\left(t,{\rm x}^{\theta_0},{\rm y}^{\theta_0}\right)^\dagger({\rm y}-\zeta),
\\
	\dot{w}_i&\geq 
	\left\langle \partial_x{\rm f}_i\left(t,{\rm x}^{\theta_0},{\rm y}^{\theta_0}\right),
	({\rm x}-\omega) \right\rangle
	+
	\left\langle\partial_y{\rm f}_i\left(t,{\rm x}^{\theta_0},{\rm y}^{\theta_0}\right),
	({\rm y}-\zeta)\right\rangle
	\\
	&=(P w+Q z)_i,
\end{aligned}\]
where ${\rm x}^\theta(t)=\theta {\rm x}(t)+(1-\theta)\omega(t)\in\R^n,$ ${\rm y}^\theta(t)=\theta{\rm y}(t)+(1-\theta) \zeta(t)\in \R^m$, denote an intermediate point for every $\theta\in [0,1]$. The usual inner product in $\R^N$ for $N=n,m$ is denoted by $\langle\cdot,\cdot\rangle$.

If we proceed analogously with $\dot{z}=f(t,{\rm x, y})-\dot{\zeta}$, then
\begin{equation}\label{inequalities}
	\left(\begin{array}{c}
 	\dot{w}
	\\
	\dot{z}
	\end{array}\right)
	\succeq
	\left(\begin{array}{c}
	P w + Q z
	\\
	R w+S z
	\end{array}\right)
	,\qquad
	\left(\begin{array}{c}
 	{w}(0)
	\\
	{z}(0)
	\end{array}\right)
	\succ\succ
		\left(\begin{array}{c}
 	\vec{0}
	\\
	\vec{0}
	\end{array}\right).
\end{equation}

We denote the primitives of 
\[
	\Delta={\rm diag}\left\{\Delta_1,\dots,\Delta_{n+m}\right\}
\]
as
\[\begin{aligned}
	\mu_h(t)&=e^{-\int^t \Delta_h}>0,\qquad h=1,\dots,n,
	\\
	\nu_h(t)&=e^{-\int^t \Delta_h}>0,\qquad h=n+1,\dots,n+m.
\end{aligned}\]
\medskip
Then the proof we are looking for will follow from the following assertion.

{\em Claim:} Each $\mu_iw_i$ is monotone increasing for $i=1,\dots,n,$ with $w_i(0)>  0$. Also each $\nu_jz_j$, $j=1,\dots,m$, is monotone decreasing with $z_j(0)< 0$.

\medskip
We now proceed this claim. We observe that,
\[
\begin{aligned}
	\frac{d}{dt}
	\left(\begin{array}{c}
		\mu \centerdot w
		\\
		\nu \centerdot z
	\end{array}\right)
	&=
	\left(\begin{array}{c}
		\mu \centerdot \dot{w}
		\\
		\nu \centerdot  \dot{z}
	\end{array}\right)
	+\left(\begin{array}{c}
		\dot{\mu}\centerdot w
		\\
		\dot{\nu}\centerdot z
	\end{array}\right)
	\\
	&\succeq
		\left(\begin{array}{c}
	\mu\centerdot(P w + Q z)
	\\
	\nu\centerdot(R w+S z)
	\end{array}\right)
	-\Delta
	\left(\begin{array}{c}
		{\mu} 
		\\
		{\nu} 
	\end{array}\right)\centerdot
	\left(\begin{array}{c}
		w
		\\
		z
	\end{array}\right)
	\\
	&= 
\left(\begin{array}{c}
		{\mu} 
		\\
		{\nu} 
	\end{array}\right)\centerdot
	\left[
	\left(\begin{array}{cc}
	P & Q
	\\
	R&S
	\end{array}\right)
	-\Delta\right]
	\left(\begin{array}{c}
		w
		\\
		z
	\end{array}\right),
\end{aligned}
\]
or
\begin{equation}\label{eq:d/dt}
	\frac{d}{dt}
	\left(\begin{array}{c}
		\mu \centerdot w
		\\
		\nu\centerdot  z
	\end{array}\right)
	\succeq
	\left(\begin{array}{c}
		{\mu} 
		\\
		{\nu} 
	\end{array}\right)
	\centerdot
	\left[
	DV(t,x,y)
	-\Delta\right]
	\left(\begin{array}{c}
		w
		\\
		z
	\end{array}\right),
\end{equation}
where $u\centerdot v\in\R^N$ for $u,v\in\R^N$, denotes the following associative product
\[
	u\centerdot v:=({\rm diag\,}u)v=(u_1v_1,\dots,u_Nv_N)\in \R^N,\, N=n,m\text{ or }n+m.
\]
Remark that the matrix $DV-\Delta$ is the same as $DV$ but with zeros along its diagonal.

Define
\[\begin{aligned}
	t_{w_i}&:= \inf\{t>0\,:\,{w}_i(t)<0\}\geq 0,\qquad
	i=1,\dots,n,
	\\
	t_{z_j}&:=\inf\{t>0\,:\,{z}_j(t)>0\}\geq 0,\qquad
	j=1,\dots,m,
\end{aligned}\]
all of them in $\R\cup\{\infty\}$. By the very definition, there are intervals
\[ 
	\dot{w}_i(t)<0,\quad
	\forall t\in(t_{w_i},t_{w_i}+\epsilon),\qquad
	\dot{z}_j(t)>0,\quad\forall t\in(t_{z_j},t_{z_j}+\epsilon).
\] 
Moreover,
\begin{equation}\label{eq:*}
	{w}_i(t_{w_i})=0,\qquad
	{z}_j(t_{z_j})=0,
\end{equation}
while 
\[
	{w}_i(t)\geq  0,\qquad
	{z}_j(t) \leq  0,
\]
for every $t\in [0,t_{w_i}]$ and $t\in [0,t_{z_j}],$ respectively. Furthermore, for $\geps>0$ small enough we can also consider $w_i(t)<0$ or $z_j(t)>0$ along those intervals, $(t_{w_i},t_{w_i}+\geps)$ and $(t_{z_j},t_{z_j}+\geps),$ respectively.

Suppose that the variables
	$v_h,$ with $h=1,\dots,n+m$
run over both sets of values $w_i,z_j$ for $i=1,\dots,n,$ and $j=1,\dots,m$ in such a way that
\[
	0<t_{v_1}\leq t_{v_2}\leq \dots\leq t_{v_{n+m}}.
\]

If $t_{v_1}=\infty$ then we are done. To prove our claim we proceed by contradiction. Suppose that
\[
	0<t_{v_1}<\infty,\qquad
	t_{v_1}\leq t_{v_{j}}\leq\infty.
\]
Then
\[
	\left(\begin{array}{c}
	w(t)
	\\
	z(t)\end{array}
	\right)\succeq \left(
	\begin{array}{c}\vec{0}
	\\
	\vec{0}\end{array}\right),
\]
Hence, when we substitute monotonicity conditions \eqref{CompH}
in \eqref{eq:d/dt}, we obtain
\[
	\left.\frac{d}{dt}\right\vert_{t\leq t_{v_1}}\left(\begin{array}{c}
		\mu \centerdot w
		\\
		\nu \centerdot z
	\end{array}\right)
	\succeq
	\left(\begin{array}{c}
		\vec{0} 
		\\
		\vec{0} 
	\end{array}\right),
		\qquad \forall t\in[0,t_{v_1}].
\]
Indeed, if for instance ${v}_{1}=w_{i_1}$ for certain index $i_1$ fixed, then for $t\in (t_{w_{i_1}},t_{w_{i_1}}+\epsilon)$, we have that $w_{i_1}(t)\geq0\leq z_{i_1}(t)$, because for every $t\in (t_{w_{i_1}},t_{w_{i_1}}+\epsilon)$,
\[\begin{aligned}
	\frac{d}{dt}\left({\mu}_{i_1}{w}_{i_1}\right)(t)
	&
	\geq
	{\mu}_{i_1}(t)\left[\sum_{k=1,\,k\neq i_1}^np_{i_1k}(t){w}_{k}(t)
	+
	\sum_{l=1}^mq_{i_1l}(t){z}_{l}(t)
	\right]\geq 0.
\end{aligned}\]
Whence, we obtain a contradiction with \eqref{eq:*}, namely,
\[
	\mu_i(t_{w_i})w_i(t_{w_i})\geq \mu_i(0)w_i(0)>0,
\]
or $w_{i_1}(t_{w_{i_1}})>0$ (not $w_{i_1}(t_{w_{i_1}})=0$).

We can reach a similar contradiction by supposing $v_1=z_j$ for some $j$ fixed.

We have thus proved that $\mu w,$ $\nu z$ are monotone increasing and decreasing respectively, with $(w(0),z(0))\succ\vec{0}$. {Therefore, $(w(t),z(t))\succ \vec{0}$ for all $t\geq 0$. This proves a couple of inequalities in \eqref{Comp-IV-1}. The other couple can be proved analogously.}}	\end{proof}

In the reversed order of $-K$ a pair sub-super-solution  if a solution $({\rm x}(t),{\rm y}(t))$ has initial condition such that
\begin{equation}\label{initial2}\begin{split}
	\omega_i(0)&>{\rm x}_i(0)>\Omega_i(0),
	\text{ and }
	Z_j(0)>{\rm y}_j(0)>\zeta_j(0),\\
	 \forall\,  i&=1,\dots,n;\quad j=1,\dots,m.
	\end{split}
\end{equation}
Then
\begin{equation}\label{Comp-IV-1-bis}
{\begin{split}
	\omega_i(t)&>{\rm x}_i(t)  > \Omega_i(t),\quad
	Z_j(t)> {\rm y}_j(t)> \zeta_j(t),\\
	\forall\,  i&=1,\dots,n;\quad j=1,\dots,m,\qquad  t\geq 0.
\end{split}}
\end{equation}

\section{Almost periodic functions}

We now recall basic concepts and facts about almost periodic functions, see \cite{Bohr, Corduneanu} for an exhaustive study.

\begin{definition}\label{ALF}
The space of {\em almost periodic functions} is the closure $\overline{\mathcal T}= \mathcal{AP}(\mathbb R,\mathbb C)$ of the algebra $\mathcal T$ of all trigonometric polynomials 
\[
	c_0+c_1e^{\lambda_1t}+\dots+c_ne^{\lambda_nt}
\]
whose frequency set, $\{\lambda_1\dots,\lambda_N\}\subset \mathbb R$ is {\em arbitrary} and $c_k\in\mathbb C$ for $k=1,\dots,N$. We consider $\mathcal T$ as a subspace of the space of bounded continuous functions $\mathcal{CB}(\mathbb R,\mathbb C)$ with the $\sup$-norm.
\end{definition}

 We just write down the main properties of the space $\mathcal{AP}(\mathbb R,\mathbb C)$:
 \begin{enumerate}[label=\roman*)]
 \item Every $\phi\in\mathcal{AP}(\R,\C)$ is uniformly continuous.
 \item $\mathcal{AP}(\R,\C)$ is a Banach algebra.
 \item For every $\phi\in\mathcal{AP}(\R,\C)$ there exists a numerable collection of frequencies $\{\lambda_k\}_{k=1}^\infty\subset\mathbb R$ whose corresponding {\em Fourier coefficients}:
 \[
 	c[\phi,\lambda_k]=\lim_{t\rightarrow\infty}\frac{1}{t-t_0}\int_{t_0}^t\phi(s)\cdot e^{-{\rm i}\lambda_ks}ds
 \]
 which do not vanish and do not depend on $t_0$. There exists an associated complex Fourier series
 \[
 	\phi(t)\sim \sum_{k=1}^\infty c[\phi,\lambda_k]e^{{\rm i}\lambda_kt}.
 \]
 \item For every $\phi\in\mathcal{AP}(\R,\C)$, there exists the {\em mean value},
 \[
 	M[\phi]=\lim_{t\rightarrow\infty}\frac{1}{t-t_0}\int_{t_0}^t\phi(s)\,ds,
 \]
 which is a well defined positive linear continuous functional, $M:\mathcal{AP}(\R,\C)\rightarrow\C$, regardless of $t_0\in\mathbb R$.
 
 \item For every $\phi\in\mathcal{AP}(\R,\C)$, the Parseval's equality holds:
 \[
 	 M\left[|\phi|^2\right]=\sum_{k=1}^\infty \left|c[\phi,\lambda_k]\right|^2.
 \]
 
 \end{enumerate}

\begin{definition}
A continuous function $f:\R\times U\rightarrow\R$ is said to be {\em uniformly almost periodic} with respect to $x\in U\subset\R^n$ if for every compact $K\subset U$,
\[
	|f(t + \tau,x) - f (t,x)| < \geps, \forall \in\R, \forall x \in U
\]
for each translation number, $\tau \in T(\geps,f,K)$, and any length $\ell(\geps,f,K) > 0$, not depending on a particular choice $x$ remaining the same on compact set $K\subset U$.
\end{definition}

More specifically, if $f$ have {\em real} Fourier expansions,
\[
	f(t,x) \sim \overline{f}(x) + \sum_{k=0}^\infty a[f, \lambda_k;x] \cos(\lambda_kt) + b[f, \lambda_k;x] \sin(\lambda_kt),
\]
then $f$ is uniformly almost periodic, whenever the Fourier frequencies $\lambda_k$ do not depend on $x$, see \cite{Corduneanu} Chapter VI.

\section{Proof of existence of an almost periodic asymptotic limit}\label{sec:results}

The following claim is a generalization to $\R^d,\, d\geq 3$ of a statement proven in \cite{Korman} for a periodic cooperative system and in \cite{Diaz-Osuna-2023} for competitive systems in $\mathbb R^2$.

\begin{theorem}[Existence and global asymptotic limit]\label{teo:Korman}
Assume that the system \eqref{Comp} has components, ${\rm f,g},$ which are uniformly almost periodic with respect to $(x,y)\in D\subset\R^n\times \R^m,$ of class ${\mathcal C}^1$ in $(x,y)$. The following assertions hold true:
\begin{enumerate}[label=\Roman*.]
\item If it is $K$-{\em intraspecific and monotone} with an almost periodic sub-super-solution pair
\[
	(\omega(t), \zeta(t))\prec\prec(\Omega(t), Z(t))
\]
that are not necessarily solutions of \eqref{Comp}. Then the system has an almost periodic solution $({x}^\star(t), {y}^\star(t))$, satisfying 
\begin{equation}\label{eqn:periodic}
(\omega(t),\zeta(t)) \preceq  ({x}^\star(t),{y}^\star(t))\preceq
(\Omega(t), Z(t)), \quad\forall t\geq 0.
\end{equation}
\item Any solution $({\rm x}(t) ,{\rm y}(t))$ of \eqref{Comp} with initial condition
\begin{equation}\label{eq:condicion-inicial}
	(\omega(0), \zeta(0))\prec\prec ({\rm x}(0) ,{\rm y}(0))
	\prec\prec(\Omega(0),Z(0)),
\end{equation}
satisfy one of the following properties:
\begin{enumerate}
\item Either $({\rm x}(t) ,{\rm y}(t))$ converges asymptotically towards an almost periodic solution $({x}^\star(t),{y}^\star(t))$,
\begin{equation}\label{eq:opcion1}
\lim_{t\rightarrow\infty}\left\|({\rm x}(t) ,{\rm y}(t))-
({x}^\star(t),{y}^\star(t))\right\|=0
\end{equation}

\item There are almost periodic solutions $({x}^\star(t),{y}^\star(t))\preceq({x}_\star(t),{y}_\star(t))$ and $T\geq 0$ such that for $t\geq T\geq 0$
\begin{equation}\label{eq:opcion2}
	\left({x}^\star(t),{y}^\star(t)\right)
	\preceq
	\left({\rm x}(t) ,{\rm y}(t)\right)
	\preceq
	\left({x}_\star(t),{y}_\star(t)\right).
\end{equation}
\end{enumerate}
\end{enumerate}
If there is {\em just one almost periodic solution} $(x(t),y(t))$ satisfying the initial condition \eqref{eq:condicion-inicial-x}. Then we conclude global asymptotic almost periodic limit, i.e. any solution $({\rm x}(t),{\rm y}(t))$ of \eqref{Comp} satisfying the initial condition \eqref{eq:condicion-inicial} converges asymptotically to the unique almost periodic attractor $(x(t),y(t))$,
\begin{equation}\label{eq:opcion3}
\lim_{t\rightarrow\infty}\left\|({\rm x}(t) ,{\rm y}(t))-
({x}(t),{y}(t))\right\|=0.
\end{equation}
\end{theorem}

\begin{remark} Notice that in the conclusion of the Main Theorem \ref{teo:Korman}
\[
	(\check{x}(t), \check{y}(t))
	\preceq ({x}^\star(t) ,{y}^\star(t))
	\preceq (\hat{x}(t), \hat{y}(t))
\]
where $ (\check{x}(t), \check{y}(t)) \preceq (\hat{x}(t), \hat{y}(t))$ are $K$-minimal and $K$-maximal almost periodic solutions of \eqref{eqn:periodic}, respectively, satisfying the initial condition
\[
	(\omega(0), \zeta(0))\prec\prec ({x}^\star(0) ,{y}^\star(0))
	\prec\prec(\Omega(0),Z(0)).
\]
Indeed, the non-empty set of almost periodic orbits $(x(t),y(t))$ such that 
\begin{equation}\label{eq:condicion-inicial-x}
	(\omega(0),\zeta(0))
	\preceq
	(x(0),y(0))
	\preceq
	(\Omega(0),Z(0)),
\end{equation}
is partially ordered by $\preceq$. Therefore, we can consider $(\check{x}(t),\check{y}(t)),$ $(\hat{x}(t),\hat{y}(t)),$ the minimal and maximal almost periodic solutions, respectively.
\end{remark}

\begin{corollary}
If the same conditions as in Theorem~\ref{teo:Korman} hold true. Assuming also that there exists just one almost periodic solution $({x}(t),{y}(t))$ with initial condition \eqref{eq:condicion-inicial-x}. If $(\omega(t), \zeta(t))\equiv(\omega_0,\zeta_0)$ and $(\Omega(t), Z(t))\equiv(\Omega_0,Z_0)$ consists of constant sub-super-solution pair with respect to the $K$-ordering. Then any solution of \eqref{Comp}, with initial condition $({\rm x}(0), {\rm y}(0))$ inside the rectangle 
\[
	(\omega_0,\zeta_0) \prec\prec
	({\rm x}(0),{\rm y}(0))
	\prec\prec (\Omega_0,Z_0),
\]
converges asymptotically as in \eqref{eq:opcion3} to the unique almost periodic solution, $(x(t),y(t))$ contained in such rectangle
\[
	(\omega_0,\zeta_0) \prec\prec
	({x}(t),{y}(t))
	\prec\prec (\Omega_0,Z_0).
\]
\end{corollary}

\subsection{Proof of Existence of an almost periodic solution in Theorem \ref{teo:Korman}}

We first prove the existence of almost periodic solutions. Take $x_0=\omega$, $y_0=\zeta$. Let $x_1(t),y_1(t)$ be {\em the unique almost periodic solution} of the following non-homogeneous linear system,

\begin{equation}\label{eqn:1}
\begin{split}
\dot{x}_1+Lx_1=L x_0 + {\rm f}(t,x_0,y_0),\\
\dot{y}_1+Ly_1=L y_0 + {\rm g}(t,x_0,y_0),
\end{split}
\end{equation}
which exists for any non-homogeneous linear ODE according to ~\cite{Corduneanu}. More specificaly,
\[
	x_{1,i}(t)=e^{-Lt}\left[c_0+\int_0^t e^{Ls}\left(Lx_0(s)+\mathrm f(z,x_0(s),y_0(s))\right)\, ds\right],
	\, i=1,\dots, n,
\]
with integration constant
\[
	c_0=-\int_0^{\infty} e^{Ls}\left(Lx_0(s)+\mathrm f(z,x_0(s),y_0(s))\right)\, ds
	\geq -\int_0^{\infty} e^{Ls}(L\omega(s)+\dot{\omega}(s))\, ds.
\]

Here $L> 0$ is a scalar to be chosen such that $L + \inf_C\{\partial_{x_k} {\rm f}_i(t,x,y)\} >0 $ and $ L +\inf_C\{\partial_{y_l} {\rm g}_j(t,x,y)\}>0$ where
\[
	 C=\{(t,x,y)\,:\, x_i\in[\omega_i(t),\Omega_i(t)],y_j\in[Z_j(t),\zeta_j(t)],\,t\in\mathbb{R} \}.
\]
We also use almost periodic uniformity with respect to $(x,y)$, see~\cite{Corduneanu}.

For each coordinate index, $i=1,\dots,n$, we have
\[
	L(x_{1,i}-\omega_i)={\rm f}_i(t,\omega,\zeta)-\dot{x}_{1,i}\geq \dot{\omega}_i - \dot{x}_{1,i}.
\]
Thus,
\begin{equation}\label{eq:alpha1}{
	\frac{d}{dt}({x}_{1,i}- \omega_i) + L(x_{1,i}-\omega_i)=
	\alpha_{1,i}(t):=
	{\rm f}_i(t,\omega(t),\zeta(t))-\dot{\omega}_i(t)\geq 0.
}\end{equation}
By Gronwall's inequality
\[
	x_{1,i}(t)-\omega_i(t)\geq (x_{1,i}(0)-\omega_i(0))e^{-Lt}\geq 0,\quad \forall t\geq 0.
\]
On the other hand, for $y_1$ we have $L(y_{1,j}-\zeta_j)={\rm g}_j(t,x_0,y_0)-\dot{y}_{1,j}\leq \dot{\zeta}_j-\dot{y}_{1,j}$ or 
\[
	\frac{d}{dt}({y}_{1,j}- \zeta_j) + L(y_{1,j}-\zeta_j)=:\beta_{1,j}(t)\leq 0.
\]
By induction we define $(x_{{\tt n}+1},y_{{\tt n}+1})$ from $(x_{{\tt n}},y_{{\tt n}})$ for every ${\tt n}\in \N,$ by taking the {\em unique almost periodic solutions} of the following linear non-homogeneous equations:
\begin{equation}\label{eq:linear-systems}\begin{aligned}
\frac{d}{dt}{x}_{{\tt n}+1}+L\,x_{{\tt n}+1}
&=
L\, x_{{\tt n}}+\,{\rm f}(t,x_{\tt n},y_{\tt n}),
\\
\frac{d}{dt}y_{{\tt n}+1}+L\, y_{{\tt n}+1}
&=
L\, y_{{\tt n}}+\,{\rm g}(t,x_{\tt n},y_{\tt n}).
\end{aligned}\end{equation}
Hence,
\[\begin{aligned}
\frac{d}{dt}({x}_{{\tt n}+1}-x_{\tt n})+L(x_{{\tt n}+1}-x_{\tt n})
&=
\alpha_{\tt n},
\\
\frac{d}{dt}(y_{{\tt n}+1}-y_{\tt n})+L(y_{{\tt n}+1}-y_{\tt n})
&=
\beta_{\tt n},
\end{aligned}\]
with
\[\begin{aligned}
\alpha_{\tt n}&=
-\dot{x}_{{\tt n}}+{\rm f}(t,x_{{\tt n}-1},y_{{\tt n}-1})+{\rm r}_{\tt n},
\qquad
{\rm r}_{\tt n}=
{\rm f}(t,x_{\tt n},y_{\tt n}) -{\rm f}(t,x_{\tt n-1},y_{\tt n-1}),
\\
\beta_{\tt n}&=
-\dot{y}_{{\tt n}}+{\rm g}(t,x_{{\tt n}-1},y_{{\tt n}-1})+{\rm s}_{\tt n},
\qquad
{\rm s}_{\tt n}=
{\rm g}(t,x_{\tt n},y_{\tt n})-{\rm g}(t,x_{\tt n-1},y_{\tt n-1}).
\end{aligned}\]
Because of \eqref{CompH}, ${\rm f}_i(t,x_{\tt n},y_{\tt n})$ is monotone increasing with respect to its components $x_{{\tt n},k}$. At the same time, it is monotone decreasing with respect to the components $y_{{\tt n},l}$.
Thus, by induction hypothesis for $k=1,\dots,n;$ and $l=1,\dots,m.$
\[
	x_{{\tt n},k}\geq x_{{\tt n}-1,k},
	\qquad
	y_{{\tt n},l}\leq y_{{\tt n}-1,l}.
\]
Therefore,
\[
	{\tt r}_{{\tt n},i}
	=
	{\rm f}_i(t,x_{{\tt n}},y_{{\tt n}})-
	{\rm f}_i(t, x_{{\tt n}-1},y_{{\tt n}})
	+
	{\rm f}_i(t, x_{{\tt n}-1},y_{{\tt n}})-
	{\rm f}_i(t, x_{{\tt n}-1},y_{{\tt n}-1})
	\geq
	0, \,
	\forall i=1,\dots,\tt n.
\]
Similarly, $	{\tt s}_{{\tt n},j}\leq 0$ for $j=1,\dots,m$. Also,
\[
-\dot{x}_{{\tt n},i}+{\rm f}_i(t,x_{{\tt n}-1},y_{{\tt n}-1})\geq 0
\]
by induction hypothesis. Thus,
\[
	\frac{d}{dt}({x}_{{\tt n}+1,i}-x_{{\tt n},i})+L(x_{{\tt n}+1,i}-x_{{\tt n},i})\geq 0,
\] and
\[
	\frac{d}{dt}(y_{{\tt n}+1,i}-y_{{\tt n},i})+L(y_{{\tt n}+1,j}-y_{{\tt n},j})\leq 0.
\]
By induction we have just proven that
\[
	x_{{\tt n}+1,i}\geq x_{{\tt n},i},\qquad
	y_{{\tt n}+1,j}\leq y_{{\tt n},j},
	\qquad i=1,\dots, n;\, j=1,\dots,m.
\]
We proceed in the same manner with the super-solutions.

We thus obtain two sequences, one of them increasing while the other one is decreasing,
\begin{equation}\label{Comp-IV}
\begin{aligned}
	(\omega,\zeta)
	\preceq(x_1,y_1)\preceq\dots (x_{\tt n},y_{\tt n})\preceq
	(X_{\tt n},Y_{\tt n})\dots \preceq
	(X_1,Y_1)\preceq(\Omega,Z).
\end{aligned}
\end{equation}
If we take the uniformly convergent sequence of almost periodic functions
\[
	{x}^\star=\lim_{{\tt n}\rightarrow \infty}x_{\tt n},
	\quad
	{y}^\star=
	\lim_{{\tt n}\rightarrow \infty}y_{\tt n},
	\qquad
	{x}_\star=\lim_{{\tt n}\rightarrow \infty}X_{\tt n},
	\quad
	{y}_\star=
	\lim_{{\tt n}\rightarrow \infty}Y_{\tt n}.
\]
Then $({x}^\star(t),{y}^\star(t)),({x}_\star(t),{y}_\star(t))$ are almost periodic. Moreover, 
\[
(\omega(t),\zeta(t))\preceq
( {x}^\star(t),{y}^\star(t))\preceq
({x}_\star(t),{y}_\star(t))
\preceq  (\Omega(t),Z(t)).
\]

Furthermore, the set $\{(x_{\tt n},y_{\tt n})\}$ solve the set of linear systems \eqref{eq:linear-systems} and
by Cauchy convergence:
\begin{equation}\label{eq:uniform-convergence}
	\lim_{{\tt n}\rightarrow \infty}x_{{\tt n},i}(t)
	={x}^\star_i(t),
\qquad
	\lim_{{\tt n}\rightarrow \infty}y_{{\tt n},j}(t)
	=
	{y}^\star_{j}(t).
\end{equation}
Thus, we have vanishing limits by continuity of ${\rm f}_i,{\rm g}_i$ and 
\[
		\lim_{{\tt n}\rightarrow\infty}{\rm r}_{{\tt n},i}=0,
		\qquad
		\lim_{{\tt n}\rightarrow\infty}{\rm s}_{{\tt n},j}=0.
\]
Therefore, $({x}^\star,{y}^\star)$ is in fact an almost periodic solution of \eqref{Comp}. Similarly, the almost periodic trajectory $({x}_\star,{y}_\star)$ becomes a solution:
\[\begin{aligned}
	\dot{x}^\star={\rm f}(t,{x}^\star,{y}^\star),\qquad
	\dot{y}^\star={\rm g}(t,{x}^\star,{y}^\star),
	\\
	\dot{x}_\star={\rm f}(t,{x}_\star,{y}_\star),\qquad
	\dot{y}_\star={\rm g}(t,{x}_\star,{y}_\star).
\end{aligned}\]
We remark, that if none of the couples of the sequence $\{(x_{\tt n},y_{\tt n})\}$ is a solution of this system, then we get strict increasing and decreasing sequences:
\[\begin{aligned}
	(\omega,\zeta)
	\prec\prec(x_1,y_1)\prec\prec\dots (x_{\tt n},y_{\tt n})\prec\prec
	(X_{\tt n},Y_{\tt n})\dots \prec\prec
	(X_1,Y_1)\prec\prec(\Omega,Z),
\end{aligned}\]

This proves the existence of almost periodic solutions.

%%%%%%%%%%%%%%%

\subsection{Proof of the asymptotic limit in Theorem \ref{teo:Korman}}

We now consider any solution $({\rm x}(t),{\rm y}(t))$ of \eqref{Comp} with initial conditions as in \eqref{initial} as follows
\begin{equation}\label{Comp-b}\begin{aligned}
	\dot{\rm x}_{i}&=
	{\rm f}_i(t,{\rm x},{\rm y}),
	\quad i=1,\dots,n,
	\\
	\dot{\rm y}_{j}&={\rm g}_j(t,{\rm x},{\rm y})	
	\quad j=1,\dots,m,
	\\
	&(\omega(0),\zeta(0))
	\prec\prec ({\rm x}(0),{\rm y}(0))
	\prec \prec(\Omega(0),Z(0)).
\end{aligned}\end{equation}
Our approach to prove assertion II in Theorem \ref{teo:Korman} is to consider the following Claim A as well as its complement $\sim {\rm A}$.

\smallskip
\noindent {\bf Claim A.} {\em Consider the uniformly converging functions $(x_{\tt n},y_{\tt n})\nearrow (x^*,y^*)$ described in \eqref{Comp-IV}. Let $({\rm x},{\rm y})$ be any solution of \eqref{Comp-b}. For every increasing time sequence, $t_{\tt n}\nearrow\infty$, there exists a sufficiently large ${\tt n}\in\mathbb{N}$ such that} 
\begin{equation}\label{eqn:<<}
  (x_{\tt n}(t_{\tt n}),y_{\tt n}(t_{\tt n}))\prec\prec({\rm x}(t_{\tt n}),{\rm y}(t_{\tt n}))
  \prec\prec(X_{\tt n}(t_{\tt n}),Y_{\tt n}(t_{\tt n})).
\end{equation}

We prove explicitly that the first inequality in \eqref{eqn:<<} holds true. The proof of the other inequality can be achieved by analogy.

Let us consider the solution $(\xi_1,\eta_1)$ with initial conditions $(\omega(0),\zeta(0))$ of the system
\[
\begin{split}
	\dot{\xi}_1+L\xi_1=
		L\omega+{\rm f}(t,\omega,\zeta),\qquad \xi_1(0)=\omega(0),\\
	\dot{\eta}_1+L\eta_1=
		L\omega+{\rm g}(t,\omega,\zeta),\qquad \eta_1(0)=\zeta(0).\\
\end{split}
\]
Notice that $(\xi_1,\eta_1)$ is not necessarily the almost periodic solution $(x_1,y_1)$ described previously in \eqref{eqn:1}, although both of them solutions of the same system.

Nonetheless, we can argue that $\dot{\xi}_{1,i}+L\xi_{1,i}\geq L\omega_i+\dot{\omega}_i$ or $\dot{\xi}_{1,i}-\dot{\omega}_i+L(\xi_{1,i}-\omega_i)\geq 0 .$ Again by Gronwall's inequality
\[
	\omega_i(t)-\xi_{1,i}(t)\leq (\omega_i(0)-\xi_{1,i}(0))e^{-Lt}\leq 0
\]
or $\xi_{1,i}(t)-\omega_i(t)\geq e^{-Lt}(\xi_{1,i}(0)-\omega_i(0))=0$. Whence,
\begin{equation}\label{eq:xia}
	\xi_{1,i}(t) \geq \omega_i(t).
\end{equation}
We remark that by definition,
\begin{equation}\label{eq:edo0}
	\frac{d}{dt}(x_{1}-\xi_1)+L(x_1-\xi_1)=0
\end{equation}
Therefore, $x_{1,i}(t)-\xi_{1,i}(t)= e^{-Lt}(x_{1,i}(0)-\omega_i(0))>0$ and
\begin{equation}\label{eq:decreasing}
	x_{1,i}>\xi_{1,i}\geq \omega_i.
\end{equation}
Besides from \eqref{eq:decreasing} and \eqref{eq:xia}, we can also conclude that
\begin{equation}\label{eq:x-xi}
	x_{1,i}(t)\searrow \xi_{1,i}(t),
	\text{ as }t\rightarrow \infty.
\end{equation}

The difference ${\rm x}_i(t)-\omega_i(t)$ is monotone increasing,
\[
	\dot{\rm x}_{i}-\dot{\omega}_i\geq
	{\rm f}_i(t,{\rm x},{\rm y})
	-{\rm f}_i(t,\omega,\zeta)\geq0,
\]
Hence, the positive initial condition ${\rm x}_{i}(0)-\omega_i(0)>0$, implies that
\[
	{\rm x}_i(t)>\omega_i(t), \quad \forall t\geq 0.
\]

\smallskip 
\noindent {\bf Claim B.} {\em There exists some $T_1>0$ such for every $t\geq T_1$
the following inequality holds,}
\begin{equation}\label{eq:cadena1}
 	 x_{1,i}(t)>	{\rm x}_{i}(t)>
	\omega_i(t),\quad \forall t\geq T_1.
\end{equation}

\begin{remark}
Notice that Claim B suffices to disprove the conclusion given by Claim A. That is ${\rm B}\Rightarrow\,  \sim {\rm A}$.
\end{remark}

Assuming Claim B, let us suppose the existence of $T_1>0$ as in \eqref{eq:cadena1}. 
On the other hand
\[\begin{aligned}
	\dot{x}_{1,i}-\dot{\rm x}_i
	+L
	(x_{1,i}-{\rm x}_i)
	<
	\dot{x}_{1,i}-\dot{\rm x}_i
	+L
	(x_{1,i}-\omega_i)
	&=
	{\rm f}_i(t,\omega,\zeta)-{\rm f}_i(t,{\rm x},{\rm y})
\leq
	0,
\end{aligned}\]
implies
\[
	x_{1,i}(t)-{\rm x}_i(t)\leq
	 e^{-L(t-T_1)}(x_{1,i}(T_1)-{\rm x}_i(T_1)),
	 \qquad
	 \forall t \geq T_1.
\]
Therefore, in addition to the decay property \eqref{eq:x-xi} we also conclude that
\begin{equation}\label{eq:cadena-1b}
	x_{1,i}\searrow {\rm x}_i,\text{ as } t\rightarrow\infty.
\end{equation}

\smallskip
\noindent {\bf Claim C.} {\em For every ${\tt n}\geq 0$ there exists $T_{{\tt n}+1}>T_0=0$ such that $T_{\tt n}$ is an increasing sequence and
\begin{equation}\label{eq:cadena-nb}
 {\rm x}_i(t) \geq x_{{\tt n},i}(t) ,
 \quad\forall t\geq T_{{\tt n}+1},
\end{equation}
 while the following generalization of inequality \eqref{eq:cadena1} holds,
\begin{equation}\label{eq:cadena-n}
	x_{{\tt n}+1,i}(t) > {\rm x}_i(t) > \omega_i(t),\quad
	\forall t\geq T_{{\tt n}+1}.
\end{equation}
Remark also that
\[ {\rm x}_i(t) \geq x_{{\tt k},i}(t) \]
whenever $ t\geq T_{{\tt n}+1}\geq T_{{\tt k}+1}$ and ${\tt k}\leq {\tt n}$}.
 \smallskip

Thus we can prove the following Lemma.

\begin{lemma}\label{lma:asymptotic1}
If we assume the conditions described in Claim C, a solution $({\rm x},{\rm y})$ of system \eqref{Comp} converges asymptotically to an almost periodic solution as in \eqref{eq:opcion1}.
\end{lemma}
\begin{proof}
By induction, let us suppose the existence of $T_{{\tt n}+1}>T_{\tt n}>0$ as in \eqref{eq:cadena-n}. Hence,
\[\begin{aligned}
x_{{\tt n},i}(t)\searrow \xi_{{\tt n},i}(t).
\end{aligned}\]
On the other hand
\[\begin{aligned}
	\dot{x}_{{\tt n}+1,i}-\dot{\rm x}_i
	+L
	(x_{{\tt n}+1,i}-{\rm x}_i)
	&<
	\dot{x}_{{\tt n}+1,i}-\dot{\rm x}_i
	+L
	(x_{{\tt n}+1,i}-x_{{\tt n},i})
	\\
	&=
	{\rm f}_i(t,x_{\tt n},y_{\tt n})-{\rm f}_i(t,{\rm x},{\rm y})
\leq
	0,
\end{aligned}\]
implies
\[
	x_{{\tt n}+1,i}(t)-{\rm x}_i(t)\leq
	 e^{-L(t-T_{{\tt n}+1})}(x_{{\tt n}+1,i}(T_1)-{\rm x}_i(T_{{\tt n}+1})),
	 \qquad
	 \forall t \geq T_{{\tt n}+1}.
\]
Therefore, we also conclude that
\[
	x_{{\tt n}+1,i}\searrow  {\rm x}_i,\text{ as } t\rightarrow\infty.
\]
Moreover, from the uniform convergence \eqref{eq:uniform-convergence}, $x_{{\tt n},i}\nearrow x_i^\star,$ as ${\tt n}\rightarrow \infty$ we prove the asymptotic limit
\[
	\lim_{{\tt n}\rightarrow \infty}{\rm x}_{i}(t)
	={x}^\star_i(t).
\]
\end{proof}

\begin{remark}
We can generalize \eqref{eq:decreasing} to conclude that
\begin{equation}\label{eq:decreasing-b}
	\omega\preceq \xi_{1} \prec\prec x_{1}.
\end{equation}
The uniform convergence as ${\tt n}\rightarrow\infty$ to solutions $(x_{\tt n},y_{\tt n})\rightarrow (x^\star,y^\star)$  and
$
	(\xi_{\tt n},\eta_{\tt n})\rightarrow (\xi^\star,\eta^\star),
$ allows to conclude that both solutions are comparable, i.e.
\begin{equation}\label{eq:cadena-n-c}
	(x^\star,y^\star)\succeq (\xi^\star,\eta^\star)
	\succ(\omega,\zeta).
\end{equation}

\end{remark}

\begin{remark}
Notice that Claim {C} $\Rightarrow$ Claim {B}. It is easy to check also that Claim C $\Rightarrow \,\sim ({\rm Claim\, A})$. Furthermore,
$\sim ({\rm Claim \,A})\Leftrightarrow {\rm Claim\,  C}.$
\end{remark}

\subsection{Conclusion of the proof of Theorem~\ref{teo:Korman}}

We now recall that \eqref{Comp-IV} states that $(x_{\tt n},y_{\tt n})$ is an increasing sequence in the partial ordering $\preceq$ which converges towards an almost periodic solution $({x}^\star,{y}^\star)$ of \eqref{Comp}.

From Lemma \ref{lma:asymptotic1} we have that Claim C implies the asymptotic limit,
\[
	({\rm x}(t),{\rm y}(t))\rightarrow ({x}^\star(t),{y}^\star(t)),\quad
	t\rightarrow\infty.
\]
Thus, conclusion II(a) in Theorem \ref{teo:Korman} arises.

On the contrary, if we assume Claim A (which happens to be the opposite of Claim C), then for every sequence $t_{\tt n}\nearrow\infty$ there exists ${\tt n}\in\mathbb N$ sufficiently large, such that
\[
	{\rm x}_{i}(t_{\tt n}) > {x}_{{\tt n},i}(t_{\tt n}).
\]
A similar property can be deduced for the other coordinates
\[
	{\rm y}_{i}(t_{\tt n}) < {y}_{{\tt n},i}(t_{\tt n}).
\]
From uniform convergence, for every $\geps>0$ there exists $N_\geps$ such that
\[
	0\leq x_i^\star(t)-x_{{\tt n},i}(t)<\geps,
	\forall t\geq 0,\quad {\tt n}\geq N_\geps
\]
So for every ${\tt n}\geq N_{\geps}$ sufficiently large,
\[
	{\rm x}_{i}(t_{\tt n})-x_i^\star(t_{\tt n})
	=
		\left[{\rm x}_{i}(t_{\tt n})-{x}_{{\tt n},i}(t_{\tt n})\right]
		-\left[x_i^\star(t_{\tt n})-{x}_{{\tt n},i}(t_{\tt n})\right]
		>
		0-\geps.
\]
Therefore, $\liminf_{t\rightarrow\infty}\left[{\rm x}_{i}(t)-x_i^\star(t)\right]\geq 0,$ or
\begin{equation}\label{eq:liminf}
	\liminf_{t\rightarrow\infty}
	{\rm x}_{i}(t)\geq
	\limsup_{t\rightarrow\infty} x_i^\star(t).
\end{equation}
Recall that by monotonicity, the pair of solutions $(\xi^\star,\eta^\star)$ and $(x^\star,y^\star)$ are comparable with respect to $(\preceq)$ as in \eqref{eq:cadena-n-c} and that by the initial condition \eqref{Comp-b},
$
	 ({\rm x}, {\rm y})\succ (\xi^\star,\eta^\star)
$. On the other hand \eqref{eq:liminf} allows us to conclude that there exists $t_{\tt n}>0$ such that $(x^\star(t_{\tt n}),y^\star(t_{\tt n}))\succeq({\rm x}(t_{\tt n}),{\rm y}(t_{\tt n}))$ Regarding these values in $T=t_{\tt n}$ as initial conditions, by monotonicity we can conclude the following comparison 
\[
	({\rm x}(t),{\rm y}(t))\succeq(x^\star(t),y^\star(t)),
	\quad \forall t\geq  T.
\]
This ends the proof of the case II(b) in Theorem~\ref{teo:Korman}.

Summarizing our results, we have proved the following assertion

\begin{lemma}
$Z^\star\in {\sf Super}\, V$ defines a super-solution $(\vec{0},Z^\star)$ while 
for $\omega^0\in{\sf sub}\, V $ we obtain a sub-solution $(\omega^0,\vec{0})$. Both are related so that we get a pair sub-super-solution.
\end{lemma}

\section{Proof of Theorem~\ref{teo:principal}}
 
 \subsection{Existence of almost periodic solutions in Theorem~\ref{teo:principal}}
 
As an application of the results exposed in the previous Section we prove Theorem~\ref{teo:principal}. Firstly, we verify the intraspecific motonicity property of system \eqref{system} which we denote by
\[
	V=\left(\begin{array}{c}
	\dot{c}_S\\ \dot{c}_I\\ \dot{c}_{\rm ES}\\ \dot{c}_{\rm EI}
	\end{array}\right).
\]
By calculating the jacobian matrix $DV$, we get
\[
	\left(\begin{array}{ccccc}
		-k_1(T-c_{\rm EI}-c_{\rm ES})-\xi_{\rm S}		&	0		& k_1c_{\rm S}+k_2	& k_1c_{\rm S}
		\\
		0				&	-k_5(T-c_{\rm ES} -c_{\rm EI})-\xi_{\rm I}	 & k_5c_{\rm I}		& k_4+k_5c_{\rm I}
		\\
		k_1(T-c_{\rm ES}-c_{\rm EI})  &0	    	& -k_1c_{\rm S} -k_2-k_3		&-k_1c_{\rm S}
		\\
		0	& k_5(T-c_{\rm ES}-c_{\rm EI})	&  -k_5c_{\rm I}  	& -k_4-k_5c_{\rm I}
\end{array}\right).
\]
It has the form of a monotone system described in \eqref{DF^T-bis}. Furthermore $DV$ has the form of an {\em intraspecific} and monotone system since
\begin{equation}\label{eq:negative}
	T-c_{\rm EI}-c_{\rm ES}	\geq 0,\qquad
\end{equation}
hold true for the dynamics inside the stoichiometric space $c\in\mathsf S\cap \mathrm R^4_{\geq 0}$. We adopt the $(-K)$-partial ordering with orthant
\[
	-K=\R^2_{\leq 0}\times\R_{\geq 0}^2,
\]
to apply our results. We define a sub-solution whose constant components,
\[
	\left(\begin{array}{c}
	\omega
	\\
	\zeta
	\end{array}\right)
	=
	\left(\begin{array}{c}
	\omega_{\rm S}\\ \omega_{\rm I} \\ \zeta_{\rm ES}
	\\ \zeta_{\rm EI}
	\end{array}\right)
\]
solve inequalities \eqref{Sup-bis}, i.e.
\[\begin{aligned}
	0
	&\geq
	(F_{\rm S})^*-k_1\omega_{\rm S}(T-\zeta_{\rm ES}-\zeta_{\rm EI})-\xi_{\rm S}\omega_{\rm S}+k_2\zeta_{\rm ES},
	\\
	0
	&\geq
	(F_{\rm I})^*-k_5\omega_{\rm I}(T- \zeta_{\rm EI}-\zeta_{\rm ES})-\xi_{\rm I}\omega_{\rm I}+k_4\zeta_{\rm EI},
	\\
	0
	&\leq
	k_1\omega_{\rm S}(T-\zeta_{\rm EI}-\zeta_{\rm ES}) - (k_2+ k_3 )\zeta_{\rm ES} ,
	\\
	0
	&\leq
	k_5\omega_{\rm I}(T-\zeta_{\rm ES}-\zeta_{\rm EI}) -k_4\zeta_{\rm EI}.
\end{aligned}\]
Solving the inequalities by imposing $\zeta=\vec{0}$, see Fig. \ref{fig:R4}, we get,
\[	
	\omega_{\rm S}\geq\omega_{\rm S}^0
	:=
	\frac{(F_{\rm S})^*}{\xi_{\rm S}+k_1T}>0,
	\quad
	\omega_{\rm I}
	\geq 
	\omega_{\rm I}^0
	:=\frac{(F_{\rm I})^*}{\xi_{\rm I}+k_5T}>0,
\quad
	\zeta_{\rm ES}=0,
	\quad
	\zeta_{\rm EI}=0.
\]
Therefore, there exists an unbounded parameter space,
$
	(u,v)\in\R^2_{\geq 0},
$
such that any $\left(\omega,\vec{0}\right)\in\R^2_{\geq 0}\times \R^2_{\geq 0},$ given by
\[
	\omega_{\rm S}=u+\omega^0_{\rm S},
	\quad
	\omega_{\rm I}=v+ \omega^0_{\rm I},
	\qquad
	u\geq 0,\, 
	v\geq 0,
\]
defines a sub-solution $(\omega,\vec{0})$ in the region ${\sf sub}\, V\subset \R_{\geq 0}^2$, see Fig. \ref{fig:sub-Super}. A vertex for ${\sf sub}\, V$ is
\[
	\omega^0=
	\left(\begin{array}{c}
		\omega^0_S\\ \omega^0_I
	\end{array}\right).
\]

\begin{figure}[t] %
 \centering
 
%%%%%%%%%%%%%%%%%%%%%%

\setlength{\unitlength}{2500sp}%
\begingroup\makeatletter\ifx\SetFigFont\undefined%
\gdef\SetFigFont#1#2#3#4#5{%
  \reset@font\fontsize{#1}{#2pt}%
  \fontfamily{#3}\fontseries{#4}\fontshape{#5}%
  \selectfont}%
\fi\endgroup%
\begin{picture}(5877,5049)(1636,-7573)
\put(6076,-5461){\makebox(0,0)[lb]{\smash{{\SetFigFont{12}{14.4}{\rmdefault}{\mddefault}{\updefault}{\color[rgb]{0,0,0}$\mathsf{S}^\circ$}%
}}}}
\thinlines
{\color[rgb]{0,0,0}\put(1801,-6961){\vector( 1, 0){5700}}
}%
{\color[rgb]{0,0,0}\put(5401,-6961){\vector(-1, 1){2925}}
}%
{\color[rgb]{0,0,0}\put(2101,-3361){\line( 1, 0){5400}}
}%
\put(7276,-7261){\makebox(0,0)[lb]{\smash{{\SetFigFont{12}{14.4}{\familydefault}{\mddefault}{\updefault}{\color[rgb]{0,0,0}$(c_{\rm S},c_{\rm I})$}%
}}}}
\put(1651,-2836){\makebox(0,0)[lb]{\smash{{\SetFigFont{12}{14.4}{\familydefault}{\mddefault}{\updefault}{\color[rgb]{0,0,0}$(c_{\rm ES}, c_{\rm EI})$}%
}}}}
\put(2701,-4111){\makebox(0,0)[lb]{\smash{{\SetFigFont{12}{14.4}{\familydefault}{\mddefault}{\updefault}{\color[rgb]{0,0,0}$({\rm 0},Z)$}%
}}}}
\put(6076,-2986){\makebox(0,0)[lb]{\smash{{\SetFigFont{12}{14.4}{\familydefault}{\mddefault}{\updefault}{\color[rgb]{0,0,0}$\R^4_{\geq0}$}%
}}}}
\put(4576,-3661){\makebox(0,0)[lb]{\smash{{\SetFigFont{12}{14.4}{\rmdefault}{\mddefault}{\updefault}{\color[rgb]{0,0,0}$c_{\rm ES}+c_{\rm EI}=T$}%
}}}}
\put(4726,-7336){\makebox(0,0)[lb]{\smash{{\SetFigFont{12}{14.4}{\familydefault}{\mddefault}{\updefault}{\color[rgb]{0,0,0}$(\omega,\vec{0})$}%
}}}}
{\color[rgb]{0,0,0}\put(2401,-7561){\vector( 0, 1){5025}}
}%
\end{picture}%

%%%%%%%%%%%%%%%%%%%%%%
   \caption{Sub-solution $(\omega,\vec{0})$ and super-solution $(\vec{0},Z)$ in the phase space $\mathbb R^4_{\geq 0}$.}
   \label{fig:R4}
\end{figure}
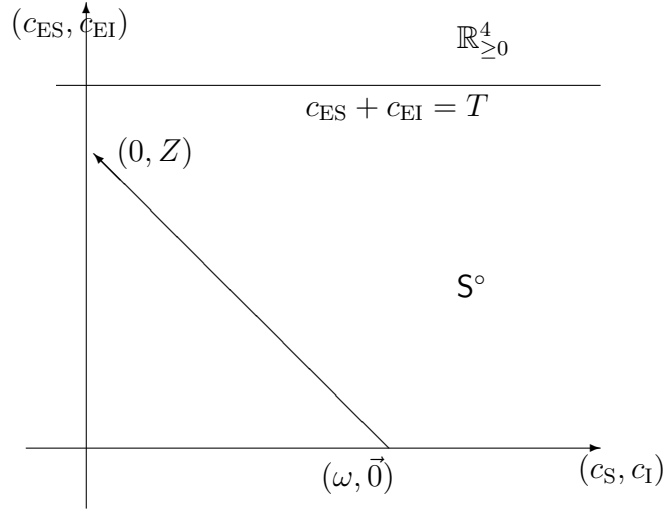
For a super-solution we take a constant vector
\[
	\left(\begin{array}{c}
	\Omega
	\\
	Z
	\end{array}\right)=
	\left(\begin{array}{c}
	\Omega_S\\ \Omega_I \\ Z_{ES} \\ Z_{EI}
	\end{array}\right)\in \mathbb R^2_{\geq 0},
\]
solving \eqref{Sub-bis}, i.e.
\[\begin{aligned}
	0
	&\leq
	(F_{\rm S})_*-k_1\Omega_{\rm S}(T-Z_{\rm ES}-Z_{\rm EI})
	-\xi_{\rm S}\Omega_{\rm S}+k_2 Z_{\rm ES},
	\\
	0
	&\leq
	(F_{\rm I})_*-k_5(T-Z_{\rm ES}-Z_{\rm EI})
	-\xi_{\rm I}\Omega_I+k_4Z_{\rm EI},
	\\
	0
	&\geq
	k_1\Omega_S(T-Z_{\rm EI}-Z_{\rm ES})
	-(k_2+k_3)Z_{\rm ES},
	\\
	0
	&\geq
	k_5\Omega_I(T-Z_{\rm EI}-Z_{\rm ES})
	-k_4Z_{\rm EI} .
\end{aligned}\]
We consider $\Omega=(\Omega_{\rm S},\Omega_{\rm I})=\vec{0}$ and the region consisting of constant super-solutions $\left(\vec{0},Z\right)\in \R^2_{\geq 0}\times\R^2_{\geq 0}$ where  $Z=(Z_{\rm ES},Z_{\rm EI})\subset \R^2_{\geq 0}$ lies within a triangle ${\sf Super}\, V\subset \R^2_{\geq0}$ limited by the lines
$
	Z_{\rm ES}=Z_{\rm ES}^0=0,\quad
	Z_{\rm EI}=Z_{\rm EI}^0=0, \quad
$
and solving the equality
\begin{equation}\label{Z-leq-1}
	0\leq Z_{\rm ES}+Z_{\rm EI}\leq T.
\end{equation}

Now we can apply Theorem \ref{teo:Korman}, by considering sub-super-solution pairs with constant coordinates
$
	\left(\omega^0,\vec{0}\right),
	\quad
	\left(\vec{0},Z^\star\right),
$ with
\[
	Z^\star=(T/2,T/2),
\]
Remarkably, $Z^\star$ is the maximum for a rectangular region $R$ (that is, $Z^\star\succeq_{\rm i} Z$ for every $Z\in R$ using the usual first quadrant, denoted as ${\rm i}$, partial ordering in $\mathbb R^2$). We will consider the following rectangle
\[
	R
	=\{(Z_{ES},Z_{EI})\in\R^2_{\geq 0}\,
	\, Z_{ES}\leq T/2,\,Z_{EI}\leq T/2
	\}
	\subset {\sf Super}\, V
\]has opposed vertices $\vec{0}\in R$ and $Z^\star\in R$. See Fig. 3.

We conclude the proof of existence of at least one 
almost periodic solution inside the parallelepiped in $\R^4$ spanned by vertices $(\vec{0},Z^\star)$ and $(\omega^0,\vec{0})$.

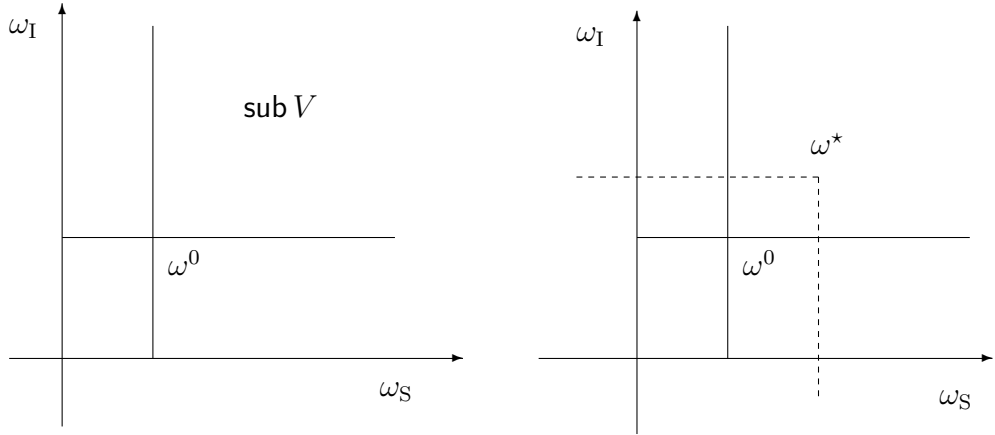
\begin{figure}[t] 
   \centering
%%%%%%%%%%%%%%%
\setlength{\unitlength}{2500sp}%
\begingroup\makeatletter\ifx\SetFigFont\undefined%
\gdef\SetFigFont#1#2#3#4#5{%
  \reset@font\fontsize{#1}{#2pt}%
  \fontfamily{#3}\fontseries{#4}\fontshape{#5}%
  \selectfont}%
\fi\endgroup%
\begin{picture}(9777,4299)(2461,-7723)
\put(10426,-4936){\makebox(0,0)[lb]{\smash{{\SetFigFont{12}{14.4}{\rmdefault}{\mddefault}{\updefault}{\color[rgb]{0,0,0}$\omega^\star$}%
}}}}
\thinlines
{\color[rgb]{0,0,0}\put(8701,-7711){\vector( 0, 1){4200}}
}%
{\color[rgb]{0,0,0}\put(9601,-6961){\line( 0, 1){3300}}
}%
{\color[rgb]{0,0,0}\put(8701,-5761){\line( 1, 0){3300}}
}%
{\color[rgb]{0,0,0}\put(3001,-7636){\vector( 0, 1){4200}}
}%
{\color[rgb]{0,0,0}\put(2476,-6961){\vector( 1, 0){4500}}
}%
{\color[rgb]{0,0,0}\put(3001,-5761){\line( 1, 0){3300}}
}%
{\color[rgb]{0,0,0}\put(3901,-6961){\line( 0, 1){3300}}
}%
{\color[rgb]{0,0,0}\multiput(10501,-7336)(0.00000,117.56757){19}{\line( 0, 1){ 58.784}}
\multiput(10501,-5161)(-117.07317,0.00000){21}{\line(-1, 0){ 58.537}}
}%
\put(9751,-6136){\makebox(0,0)[lb]{\smash{{\SetFigFont{12}{14.4}{\rmdefault}{\mddefault}{\updefault}{\color[rgb]{0,0,0}$\omega^0$}%
}}}}
\put(11701,-7411){\makebox(0,0)[lb]{\smash{{\SetFigFont{12}{14.4}{\rmdefault}{\mddefault}{\updefault}{\color[rgb]{0,0,0}$\omega_{\rm S}$}%
}}}}
\put(8101,-3811){\makebox(0,0)[lb]{\smash{{\SetFigFont{12}{14.4}{\rmdefault}{\mddefault}{\updefault}{\color[rgb]{0,0,0}$\omega_{\rm I}$}%
}}}}
\put(4051,-6136){\makebox(0,0)[lb]{\smash{{\SetFigFont{12}{14.4}{\rmdefault}{\mddefault}{\updefault}{\color[rgb]{0,0,0}$\omega^0$}%
}}}}
\put(2476,-3736){\makebox(0,0)[lb]{\smash{{\SetFigFont{12}{14.4}{\rmdefault}{\mddefault}{\updefault}{\color[rgb]{0,0,0}$\omega_{\rm I}$}%
}}}}
\put(6151,-7336){\makebox(0,0)[lb]{\smash{{\SetFigFont{12}{14.4}{\rmdefault}{\mddefault}{\updefault}{\color[rgb]{0,0,0}$\omega_{\rm S}$}%
}}}}
\put(4801,-4561){\makebox(0,0)[lb]{\smash{{\SetFigFont{12}{14.4}{\rmdefault}{\mddefault}{\updefault}{\color[rgb]{0,0,0}${\sf sub}\, V$}%
}}}}
{\color[rgb]{0,0,0}\put(7726,-6961){\vector( 1, 0){4500}}
}%
\end{picture}%
%%%%%%%%%%%%%%%%
   \caption{A set of sub-solutions, ${\sf sub}\,V\subset \R^2_{\geq 0}$. It has vertex $\omega^0=(\omega^0_{\rm S},\omega_{\rm I}^0)$. A suitable sub-solution $\omega^\star\in {\sf sub}\, V$}.
   \label{fig:sub-Super}
%%%%%%%%%%%%%%%%
\end{figure}
%%%%%%%%%%%%%%%%

%%%%%%%%%%%%%%%%
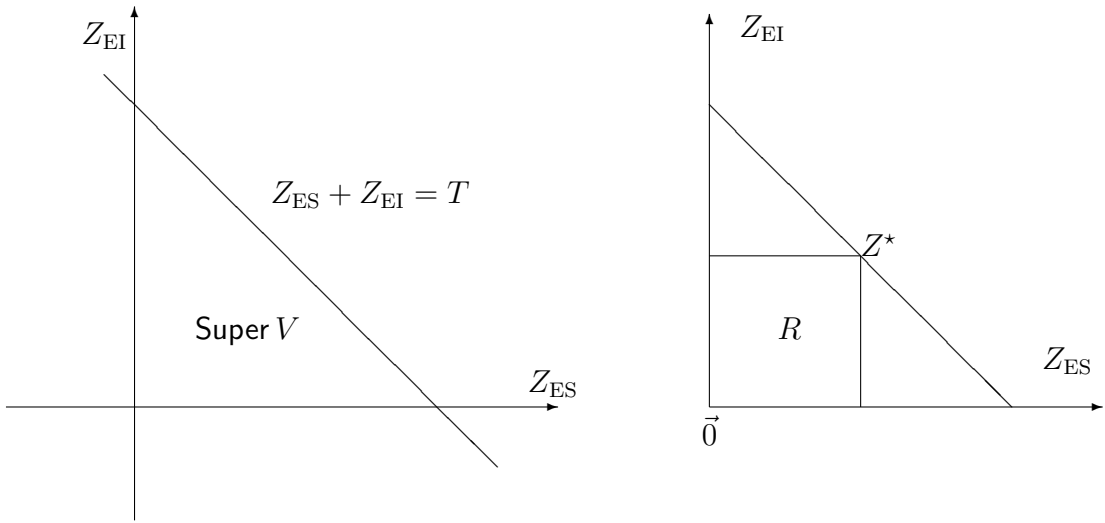
\begin{figure}
%%%%%%%%%%%%%%%%
\setlength{\unitlength}{2500sp}%
\begingroup\makeatletter\ifx\SetFigFont\undefined%
\gdef\SetFigFont#1#2#3#4#5{%
  \reset@font\fontsize{#1}{#2pt}%
  \fontfamily{#3}\fontseries{#4}\fontshape{#5}%
  \selectfont}%
\fi\endgroup%
\begin{picture}(11352,5124)(661,-7498)
\put(8776,-5686){\makebox(0,0)[lb]{\smash{{\SetFigFont{12}{14.4}{\rmdefault}{\mddefault}{\updefault}{\color[rgb]{0,0,0}$R$}%
}}}}
\thinlines
{\color[rgb]{0,0,0}\put(8101,-6361){\vector( 1, 0){3900}}
}%
{\color[rgb]{0,0,0}\put(1126,-6361){\vector( 1, 0){5475}}
}%
{\color[rgb]{0,0,0}\put(2401,-7486){\vector( 0, 1){5100}}
}%
{\color[rgb]{0,0,0}\put(8101,-3361){\line( 1,-1){3000}}
}%
{\color[rgb]{0,0,0}\put(2101,-3061){\line( 1,-1){3900}}
}%
{\color[rgb]{0,0,0}\put(8101,-4861){\line( 1, 0){1500}}
\put(9601,-4861){\line( 0,-1){1500}}
}%
\put(676,-2761){\makebox(0,0)[lb]{\smash{{\SetFigFont{12}{14.4}{\rmdefault}{\mddefault}{\updefault}{\color[rgb]{0,0,0}}%
}}}}
\put(11401,-5986){\makebox(0,0)[lb]{\smash{{\SetFigFont{12}{14.4}{\rmdefault}{\mddefault}{\updefault}{\color[rgb]{0,0,0}$Z_{\rm ES}$}%
}}}}
\put(8401,-2686){\makebox(0,0)[lb]{\smash{{\SetFigFont{12}{14.4}{\rmdefault}{\mddefault}{\updefault}{\color[rgb]{0,0,0}$Z_{\rm EI}$}%
}}}}
\put(9601,-4861){\makebox(0,0)[lb]{\smash{{\SetFigFont{12}{14.4}{\rmdefault}{\mddefault}{\updefault}{\color[rgb]{0,0,0}$Z^\star$}%
}}}}
\put(8026,-6736){\makebox(0,0)[lb]{\smash{{\SetFigFont{12}{14.4}{\rmdefault}{\mddefault}{\updefault}{\color[rgb]{0,0,0}$\vec{0}$}%
}}}}
\put(6301,-6211){\makebox(0,0)[lb]{\smash{{\SetFigFont{12}{14.4}{\rmdefault}{\mddefault}{\updefault}{\color[rgb]{0,0,0}$Z_{\rm ES}$}%
}}}}
\put(1876,-2761){\makebox(0,0)[lb]{\smash{{\SetFigFont{12}{14.4}{\rmdefault}{\mddefault}{\updefault}{\color[rgb]{0,0,0}$Z_{\rm EI}$}%
}}}}
\put(3751,-4336){\makebox(0,0)[lb]{\smash{{\SetFigFont{12}{14.4}{\rmdefault}{\mddefault}{\updefault}{\color[rgb]{0,0,0}$Z_{\rm ES} + Z_{\rm EI} = T$}%
}}}}
\put(3001,-5686){\makebox(0,0)[lb]{\smash{{\SetFigFont{12}{14.4}{\rmdefault}{\mddefault}{\updefault}{\color[rgb]{0,0,0}${\sf Super}\, V$}%
}}}}
{\color[rgb]{0,0,0}\put(8101,-6361){\vector( 0, 1){3900}}
}%
\end{picture}%

%%%%%%%%%%%%%%%%
   \caption{Region ${\sf Super}\,V\subset \R^2_{\geq 0}$ parametrizing a set of super-solutions. A maximum for a suitable rectangle $R\subset {\sf Super}\,V$ is attained at the vertex of $R$ opposed to $\vec{0}$, $Z^\star=(Z^\star_{\rm S},Z^\star_{\rm I})\in {\sf Super}\,V$.}
   \label{fig:Super}
%%%%%%%%%%%%%%%%
\end{figure}
%%%%%%%%%%%%%%%%

\subsection{Uniqueness of the almost periodic solution in Theorem~\ref{teo:principal}}

\begin{lemma}
The region 
\[
	U^\star=\{0\leq\omega_{\rm S}\leq \omega^\star_{\rm S},
	0\leq\omega_{\rm I}\leq \omega^\star_{\rm I}\}\times 
	{\sf Super}\, V	\subset\mathbb R^4_{\geq 0}
\]
where
\[
	\omega^\star_{\rm S}=\frac{Tk_2+(F_{\rm S})^*}{\xi_{\rm S}},
	\quad
	\omega^\star_{\rm I}=\frac{Tk_4+(F_{\rm I})^*}{\xi_{\rm I}},
\]
is an attractor.
\end{lemma}

\begin{proof}
We consider first $\omega,Z\in \R^2_{\geq 0}$. If $T -{Z}_{\rm ES}-{Z}_{\rm EI}<0$, then
\[
	\dot{Z}_{\rm ES}+\dot{Z}_{\rm EI}
	=
	(T-Z_{\rm ES}-Z_{\rm EI})(k_1\omega_{\rm S}+k_5\omega_{\rm I})
	-(k_2+k_3)Z_{\rm ES}-k_4Z_{\rm EI}<0.
\]
Hence $ \mathbb R^2_{\geq 0}\times \{(Z_{\rm ES},Z_{\rm EI})\,:\,0\leq Z_{\rm ES}+Z_{\rm EI}\leq  T\}$ is an attractive region for solutions. For a point in the complement,
\[\begin{aligned}
	 \omega_{\rm S}\geq   \omega_{\rm S}^{\star} ,\quad  
	 \omega_{\rm I}\geq \omega_{\rm I}^\star,\quad
	Z_{\rm ES}+Z_{\rm EI}\geq T,
\end{aligned}\]
we have decreasing linear functions $\omega_{\rm S}$ and $\omega_{\rm I}$ along the solutions for, i.e.
\[\begin{aligned}
	\dot{\omega}_{\rm S}
	&= 
	-k_1\omega_{\rm S}(T-Z_{\rm ES}-Z_{\rm EI})+k_2Z_{\rm ES}+F_{\rm S}-\xi_{\rm S}\omega_{\rm S}
	\\&\leq
	-k_1\omega^\star_{\rm S}(T-Z_{\rm ES}-Z_{\rm EI})+k_2Z_{\rm ES}+F_{\rm S}-\xi_{\rm S}\omega^\star_{\rm S}
	\\
	&\leq 
	-\omega^\star_{\rm S}\left(k_1(T-Z_{\rm ES}-Z_{\rm EI})+\xi_{\rm S}\right)
	+k_2T+(F_{\rm S})^*
	\\
	&\leq 
	-\omega^\star_{\rm S}\xi_{\rm S}
	+k_2T+(F_{\rm S})^*
	\\
	&\leq 
	-\frac{Tk_2+(F_{\rm S})^*}{\xi_{\rm S}}\xi_{\rm S}
	+k_2T+(F_{\rm S})^*
	\\
	&\leq 
	0,
\end{aligned}\]
A similar argument can be shown to $\dot{\omega}_{\rm I}$. Thus,
\[\begin{aligned}
\dot{\omega}_{\rm S}\leq
	0,
\quad
\dot{\omega}_{\rm I}
	\leq 0.
\end{aligned}\]
Therefore, $U^\star$ is an attractive region. \end{proof}

Now we can prove that the region
\[
	R=\{(Z_{\rm ES},Z_{\rm EI})\,:\, 0\leq Z_{\rm ES}\leq T,\, 0\leq Z_{\rm EI}\leq T/2\}
	\subset \mathsf{Super}\, V,
\]
which has a maximum $Z^\star$ (see Fig. \ref{fig:Super}), defines a suitable positive invariant region.

\begin{lemma}
The region $U\subset \mathbb R_{\geq 0}^4$ defined as
\[
	U :\qquad
	0\leq \omega_{\rm S}\leq \omega^\star_{\rm S},
	\quad
	0\leq \omega_{\rm I}\leq \omega^\star_{\rm I},
	\quad
	0\leq Z_{\rm ES}\leq T/2,
	\quad
	0\leq Z_{\rm EI}\leq T/2,
\]
is an attractor and positively invariant.

\end{lemma}
\begin{proof}
We evaluate the vector field $V$ along the different faces of the boundary, $\partial U\cap \partial C_i$ where $C_i$ are polytopes in the neighborhood of $U$, $i=1,\dots 8$, to verify that $V$ points towards the interior of such region. 

For the different faces of $U$ we get the corresponding inequalities as follows
\[\begin{aligned}
	C_1 :\quad &
	 0\leq \omega_{\rm S}\leq \omega_{\rm S}^\star,
	 \, 0\leq \omega_{\rm I}\leq \omega_{\rm I}^\star,
	\,  Z_{\rm ES} > T/2,
	\,  0\leq Z_{\rm EI}\leq T/2<0,
\\
	&\dot{Z}_{\rm ES}\vert_{C_1}
	\leq
	-(k_2+k_3)Z_{\rm ES},
\end{aligned}\]
\[\begin{aligned}
	C_2 :\quad &
	 0\leq \omega_{\rm S}\leq \omega_{\rm S}^\star,
	 \, 0\leq \omega_{\rm I}\leq \omega_{\rm I}^\star,
	\,  0\leq Z_{\rm ES}\leq T/2,
	\,  Z_{\rm EI}>T/2,
\\
	&\dot{Z}_{\rm EI}\vert_{C_2}
	\leq
	-k_4Z_{\rm EI}<0,
\end{aligned}\]

\[\begin{aligned}
	C_3:\quad &
	 \omega_{\rm S}<0,
	\,
	0\leq  \omega_{\rm I}\leq \omega_{\rm I}^\star,
	\, 0\leq Z_{\rm ES}\leq T/2,\, 0\leq Z_{\rm EI}\leq T/2,
\\
	&\dot{\omega_{\rm I}}\vert_{C_3}
	>
	k_2Z_{\rm ES}+F_{\rm S}\geq 0
	,
\end{aligned}\]

\[\begin{aligned}
	C_4:\quad &
	0\leq \omega_{\rm S}\leq \omega_{\rm S}^\star,
	\,
	 0> \omega_{\rm I},
	\, 0\leq Z_{\rm ES}\leq T/2,\, 0\leq Z_{\rm EI}\leq T/2,
\\
	&\dot{\omega_{\rm I}}\vert_{C_4}
	>
	k_4Z_{\rm EI}+F_{\rm I}\geq 0
	,
\end{aligned}\]

\[\begin{aligned}
	C_5:\quad &
	 \omega_{\rm S}>\omega_{\rm S}^\star,
	\,
	0\leq  \omega_{\rm I}\leq\omega^\star_{\rm I},
	\, 0\leq Z_{\rm ES}\leq T/2,\, 0\leq Z_{\rm EI}\leq T/2,
\\
	&\dot{\omega_{\rm S}}\vert_{C_5}
	< 
	 k_2 Z_{\rm ES}+F_{\rm S}-\xi_{\rm S}\omega^\star_{\rm S}
\\	
	&\leq 
	k_2T/2-\xi_{\rm S}\frac{k_2T+(F_{\rm S})^*}{\xi_{\rm S}}  +(F_{\rm S})^*
\\
	&
	\leq - k_2T/2\leq 0.
\end{aligned}\]

\[\begin{aligned}
	C_6:\quad &
	0\leq \omega_{\rm S}\leq \omega_{\rm S}^0,
	\,
	 \omega_{\rm I}>\omega^0_{\rm I},
	\, 0\leq Z_{\rm ES}\leq T/2,\, 0\leq Z_{\rm EI}\leq T/2,
\\
	&\dot{\omega_{\rm I}}\vert_{C_6}
	<
	k_4 Z_{\rm EI}+F_{\rm I}-\xi_{\rm I}\omega^\star_{\rm I}
\\	
	&\leq 
	k_4 T/2+(F_{\rm I})^*-\xi_{\rm I} \frac{k_4T+(F_{\rm I})^*}{\xi_{\rm I}} 
	\\
	&
	\leq -k_4 T/2\leq 0.
\end{aligned}\]

The verification in the faces $Z_{\rm ES}=0$ and $Z_{\rm EI}=0$ is similar. \end{proof}

So, let us consider the set of positive almost periodic solutions within the positively invariant region $U$ and for which $(\omega^\star,\vec{0})$ is a sub-solution while $(\vec{0},Z^\star)$ is a super-solution. Remark that $(\omega^\star,\vec{0})\preceq (\omega^0,\vec{0})$, see Fig.~\ref{fig:sub-Super}.

Let us denote the minimal and maximal solutions of such set as,
\[
	\check{c}=\left(\begin{array}{c}
		\check{c}_{\rm S} \\
		\check{c}_{\rm I}\\ \check{c}_{\rm EI} \\ \check{c}_{\rm ES}
	\end{array}\right),
	\qquad
		\hat{c}=\left(\begin{array}{c}
		\hat{c}_{\rm S}\\ \hat{c}_{\rm I}\\ \hat{c}_{\rm EI}\\ \hat{c}_{\rm ES}
	\end{array}\right),
\]
respectively. For the sake of brevity we denote 
\[
	\delta_{\rm A}=M\left[\hat{c}_{\rm A}-\check{c}_{\rm A}\right],
	\qquad
	\delta_{\rm A,B}=M\left[
		\hat{c}_{\rm A}\hat{c}_{\rm B}
	-\check{c}_{\rm A}\check{c}_{\rm B}\right].
\] 
Then, by subtracting and by taking the mean values in \eqref{system} we get
\begin{subequations}
\begin{equation}\label{eq:deltaS}
 0 =
 	-k_1T\delta_{\rm S}+k_1\delta_{\rm ES,S}+k_1\delta_{\rm EI,S}-\xi_{\rm S}\delta_{\rm S}
 \end{equation}
 \begin{equation}\label{eq:deltaI}
0 =
 	-k_5T\delta_{\rm I}+k_5\delta_{\rm ES,I}+k_5\delta_{\rm EI,S}-\xi_{\rm I}\delta_{\rm I}
 \end{equation}
 \begin{equation}\label{eq:deltaES}
   0 =
	k_1T \delta_{\rm S}
	-k_1\delta_{\rm ES,S}
	-k_1\delta_{\rm EI,S}-(k_2+k_3)\delta_{\rm ES},
 \end{equation}
 \begin{equation}\label{eq:deltaEI}
   0 =
	k_5T \delta_{\rm I}
	-k_5\delta_{\rm ES,I}
	-k_5\delta_{\rm EI,I}-k_4\delta_{\rm EI}.
\end{equation}
\end{subequations}

Our aim to attain the proof of uniqueness is to apply a well known result, written as Lemma \ref{lma:parseval} below. In order to do so we need to prove that
\[
	\delta_{\rm S}=\delta_{\rm I}=\delta_{\rm ES}=\delta_{\rm EI}=0.
\]
to conclude that
\[
	\hat{c}_{\rm S}(t)=\check{c}_{\rm S}(t),\quad
	\hat{c}_{\rm I}(t)=\check{c}_{\rm I}(t),\quad
	\hat{c}_{\rm ES}(t)=\check{c}_{\rm ES}(t),\quad
	\hat{c}_{\rm EI}(t)=\check{c}_{\rm EI}(t).
\]
We notice that by the partial ordering we get the following signed mean values,
\[
	\delta_{\rm S}\leq0,\quad
	\delta_{\rm I}\leq0,\quad
	\delta_{\rm ES}\geq0,\quad
	\delta_{\rm EI}\geq0.
\]
By adding \eqref{eq:deltaS} to \eqref{eq:deltaS} and \eqref{eq:deltaI} to \eqref{eq:deltaEI} we obtain
\begin{subequations}\label{eq:dsdI}
\begin{equation}\label{eq:dS+dES}
	\xi_{\rm S}\delta_{\rm S} + (k_2+k_3)\delta_{\rm ES}=0
\end{equation}
\begin{equation}\label{eq:dI+dEI}
	\xi_{\rm I}\delta_{\rm I} + k_4\delta_{\rm EI}=0
\end{equation}
\end{subequations}
Substitution of \eqref{eq:dS+dES} into \eqref{eq:deltaS} and of \eqref{eq:dI+dEI} into \eqref{eq:deltaI} yields
\[
	k_1(\delta_{\rm ES,S}+\delta_{\rm EI,S})
	=(k_1T+\xi_ {\rm S})\delta_{\rm S}\leq 0, 
	\quad
	k_5(\delta_{\rm ES,I}+\delta_{\rm EI,I})
	=(k_5T+\xi_{\rm I})\delta_{\rm I}\leq 0.
\]
Therefore,
\[
	\delta_{\rm ES,S}+\delta_{\rm EI,S}\leq 0,
	\quad
	\delta_{\rm ES,I}+\delta_{\rm EI,I}\leq 0,
\]
which allows us to derive the following inequalities,
\begin{equation}\label{eq:penultima}
\begin{aligned}
	0&\leq
	M\left[
		\left(	\hat{c}_{\rm ES}+\hat{c}_{\rm EI}\right)\hat{c}_{\rm S}
	\right]
	\leq
	M\left[
		\left(	\check{c}_{\rm ES}+\check{c}_{\rm EI}\right)\check{c}_{\rm S}
	\right]
	\\
	0&\leq
	M\left[
		\left(	\hat{c}_{\rm ES}+\hat{c}_{\rm EI}\right)\hat{c}_{\rm I}
	\right]
	\leq
	M\left[
		\left(	\check{c}_{\rm ES}+\check{c}_{\rm EI}\right)\check{c}_{\rm I}
	\right]	
\end{aligned}\end{equation}
On the other hand, we know from the ordering relations that
$
		\check{c}_{\rm ES}+\check{c}_{\rm EI}\leq 
		\hat{c}_{\rm ES}+\hat{c}_{\rm EI}
$ or 
\begin{equation}\label{eq:ultima}
\begin{aligned}
		(\check{c}_{\rm ES}+\check{c}_{\rm EI})
		\check{c}_{\rm I}
		&\leq 
		(\hat{c}_{\rm ES}+\hat{c}_{\rm EI})
		\check{c}_{\rm I}
		\\
		(\check{c}_{\rm ES}+\check{c}_{\rm EI})
		\check{c}_{\rm S}
		&\leq 
		(\hat{c}_{\rm ES}+\hat{c}_{\rm EI})
		\check{c}_{\rm S}
\end{aligned}	
\end{equation}
If we take the mean value of \eqref{eq:ultima} and combine it with \eqref{eq:penultima} we conclude that
\[\begin{aligned}
	M\left[
		\left(	\hat{c}_{\rm ES}+\hat{c}_{\rm EI}\right)\hat{c}_{\rm S}
	\right]
	&\leq
	M\left[
		\left(	\hat{c}_{\rm ES}+\hat{c}_{\rm EI}\right)\check{c}_{\rm S}
	\right]
	\\
	M\left[
		\left(	\hat{c}_{\rm ES}+\hat{c}_{\rm EI}\right)\hat{c}_{\rm I}
	\right]
	&\leq
	M\left[
		\left(	\hat{c}_{\rm ES}+\hat{c}_{\rm EI}\right)\check{c}_{\rm I}
	\right]	
\end{aligned}\]
or
\[\begin{aligned}
	0
	&\leq
	M\left[
		\left(	\hat{c}_{\rm ES}+\hat{c}_{\rm EI}\right)
		(\check{c}_{\rm S}-\hat{c}_{\rm S})
	\right]\leq 0
	\\
	0
	&\leq
	M\left[
		\left(	\hat{c}_{\rm ES}+\hat{c}_{\rm EI}\right)
		(\check{c}_{\rm I}-\hat{c}_{\rm I})
	\right]\leq 0
\end{aligned}\]
Thus
\[
		\left(	\hat{c}_{\rm ES}+\hat{c}_{\rm EI}\right)
		(\check{c}_{\rm I}-\hat{c}_{\rm I})=0
		=
		\left(	\hat{c}_{\rm ES}+\hat{c}_{\rm EI}\right)
		(\check{c}_{\rm S}-\hat{c}_{\rm S})
\]

If $\hat{c}_{\rm ES}+\hat{c}_{\rm EI}>0$ then 
\[
	\delta_{\rm S}=0=\delta_{\rm I},
\]
which in its turn by \eqref{eq:dsdI} implies the remaining identities $\delta_{\rm ES}=0=\delta_{\rm EI}$.
Otherwise, $\hat{c}_{\rm ES}+\hat{c}_{\rm EI}=0$, then $0\leq\hat{c}_{\rm ES}= -\hat{c}_{\rm EI}\leq 0$. Hence,
\[
	\hat{c}_{\rm ES}=\hat{c}_{\rm EI}.
\]
Analogous reasonings, yield $\check{c}_{\rm ES}=\check{c}_{\rm EI}$. Thus $\delta_{\rm ES}=0=\delta_{\rm EI}$.

We proceed by further simplifications to conclude that
\[
	{\hat{c}_{\sigma}}= {\check{c}_{\sigma}},
\quad \forall \sigma\in\{{\rm S,I,ES,EI}\}.
\]
This proves the uniqueness claim.

\begin{lemma}\label{lma:parseval}
Let $\hat{\phi},\check{\phi}$ be almost periodic functions such that 
\[
	\hat{\phi}(t)\geq\check{\phi}(t)\geq 0,\qquad
	{M}\left[\hat{\phi}\right]= {M}\left[\check{\phi}\right].
\]
Then $\hat{\phi}(t)=\check{\phi}(t)$ for every $t\in\R$.
\end{lemma}

\section{Numerical examples}\label{sec:examples}

We adopt the following forcing terms
\[
	F_{\rm S}(t)=1+\cos{t},
	\qquad
	F_{\rm I}(t)=1+\sin (\pi t),
\]
which are not synchronized and have upper bounds 
\[
	(F_{\rm S})^*=1.5, \,(F_{\rm I})^*=1.5,\,
	(F_{\rm S})_*=0, \,(F_{\rm I})_*=0.
\]
For the decay rate values we take  $\xi_{\rm I}=\xi_{\rm S}=1$. Finally,
\[
	k_1=0.95,\quad k_2=0.3,\quad k_3=0.9,
	\quad k_4=0.8,\quad k_5=0.3.
\]
Numerical evidence for $T=1$ in Fig. \ref{fig:No} shows an almost periodic global attractor.

\begin{figure}[h]
\begin{center}
\includegraphics[width=2.5in]{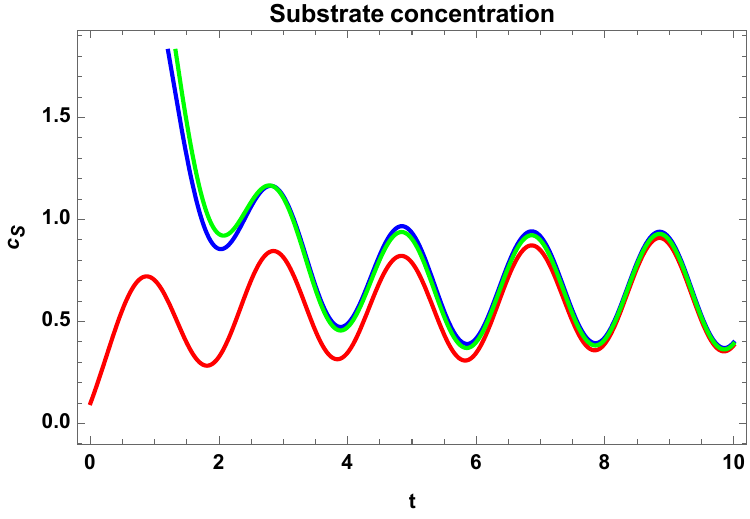}
\includegraphics[width=2.5in]{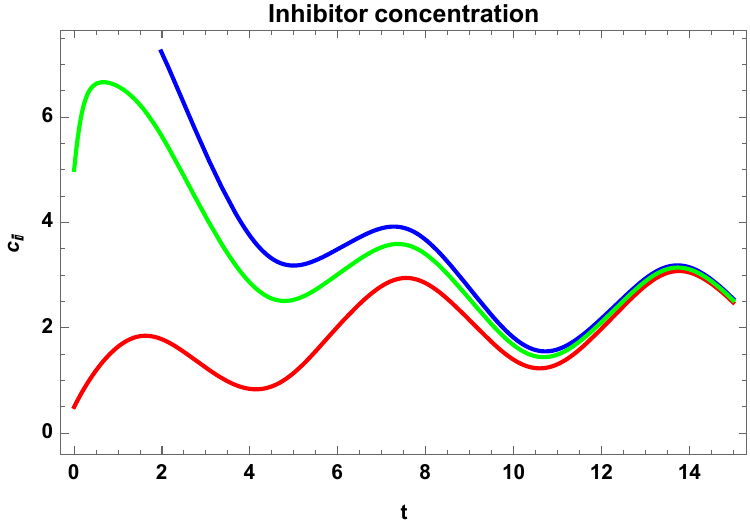}
\caption{An almost periodic solution $(c_{\rm S},c_{\rm I})$ attained as asymptotic limit by several initial conditions.}
\label{fig:No}
\end{center}
\end{figure}

\section{Discussion}

We have proved global stability of almost periodic solutions in enzyme catalysis for a specific reactor having almost periodic substrate and inhibitor supplies. This could be regarded as a first case study, i.e. just as a step in the path for the search of the most general global stability statement for a wider class of dynamical systems, namely {\em intraspecific and monotone open reaction networks.} {Thus, as we have mentioned earlier in the introduction, this work belongs to a series of articles sketching a program that addresses global stability for intraspecific monotone class of open reaction networks. The mass-action law case will be extended in \cite{DOS-1}} for other general kinetics using the tools we have developed here. We can also explore in a future work a categorical formalism, derived from the constructions exposed in \cite{Baez-2017,Baez-2018,Baez-2020} and \cite{Feinberg-2019}, for such class of monotone open systems.

\bibliographystyle{plain}
\bibliography{almost-periodic,Enzy}

\end{document}